\documentclass[11pt]{article}
\usepackage[T1]{fontenc}
\usepackage[utf8]{inputenc}
\usepackage{lmodern}
\usepackage[margin=1.15in]{geometry}
\usepackage{amsmath,amssymb,amsthm,mathtools,aliascnt}
\usepackage{microtype}
\usepackage{graphicx}
\usepackage{booktabs,tabularx}
\usepackage{enumitem}
\usepackage[font=small,labelfont=bf]{caption}
\usepackage[round,authoryear]{natbib}
\usepackage{xcolor}
\usepackage{xurl}
\usepackage[hypertexnames=false,colorlinks=true,linkcolor=blue!45!black,citecolor=blue!45!black,urlcolor=blue!45!black]{hyperref}
\usepackage[capitalise,nameinlink,noabbrev]{cleveref}
\usepackage{needspace}
\DeclareMathOperator{\dist}{dist}
\DeclareMathOperator{\gspan}{gspan}
\DeclareMathOperator{\arccosh}{arccosh}
\DeclareMathOperator{\arcsinh}{arcsinh}
\DeclareMathOperator{\interior}{int}

\DeclareMathOperator{\spanof}{span}
\DeclareMathOperator{\Proj}{Proj}
\newcommand{\calM}{\mathcal M}
\newcommand{\calF}{\mathcal F}
\newcommand{\calA}{\mathcal A}
\newcommand{\calO}{\mathcal O}
\newcommand{\HH}{\mathbb H}
\newcommand{\RR}{\mathbb R}
\newcommand{\T}{\mathrm T}
\newcommand{\xref}{x_{\mathrm{ref}}}
\newcommand{\inner}[2]{\langle #1,#2\rangle}
\newcommand{\innerL}[2]{\langle #1,#2\rangle_{\mathrm L}}
\newcommand{\norm}[1]{\lVert #1\rVert}

\newcommand{\positive}[1]{[#1]_+}
\newcommand{\eps}{\epsilon}

\newtheorem{theorem}{Theorem}[section]
\newaliascnt{proposition}{theorem}
\newtheorem{proposition}[proposition]{Proposition}
\aliascntresetthe{proposition}
\newaliascnt{lemma}{theorem}
\newtheorem{lemma}[lemma]{Lemma}
\aliascntresetthe{lemma}
\newaliascnt{corollary}{theorem}

\aliascntresetthe{corollary}
\theoremstyle{definition}
\newaliascnt{definition}{theorem}

\aliascntresetthe{definition}
\newaliascnt{assumption}{theorem}

\aliascntresetthe{assumption}
\theoremstyle{remark}
\newaliascnt{remark}{theorem}
\newtheorem{remark}[remark]{Remark}
\aliascntresetthe{remark}
\newtheorem*{remark*}{Remark}
\AddToHook{env/theorem/before}{\Needspace{4\baselineskip}}
\AddToHook{env/lemma/before}{\Needspace{4\baselineskip}}
\AddToHook{env/proposition/before}{\Needspace{4\baselineskip}}
\setlist{topsep=4pt,itemsep=3pt,parsep=0pt}
\newcommand{\insight}[2]{\begin{center}\fbox{\begin{minipage}{0.93\linewidth}\textbf{#1.} #2\end{minipage}}\end{center}}
\title{Tight lower bounds for Lipschitz geodesically convex\\ optimization on hyperbolic space}
\author{Christopher Criscitiello\thanks{The Wharton School, University of Pennsylvania, USA. \texttt{ccriscitiello6@gmail.com}.}}
\date{October 19, 2026}
\hypersetup{pdftitle={Tight lower bounds for Lipschitz geodesically convex optimization on hyperbolic space},pdfauthor={Christopher Criscitiello}}
\begin{document}
\maketitle
\begin{abstract}
We determine the dimension-independent first-order query complexity of minimizing a Lipschitz geodesically convex (g-convex) function on hyperbolic space. Given a ball of radius $r$ containing a global minimizer, we prove that $\Omega(\zeta_r/\eps^2)$ queries are necessary to attain accuracy $\eps Mr$, where $M$ is the Lipschitz constant and $\zeta_r=r/\tanh r$ captures the effect of curvature. This matches the known $O(\zeta_r/\eps^2)$ upper bound for the Riemannian subgradient method \citep{ZS16}. This is the first query-complexity result for g-convex optimization that is tight simultaneously in curvature and target accuracy.

The standard Euclidean recipe of constructing hard instances as maxima of affine functions does not transfer to hyperbolic space, so the lower bound requires a new mechanism. Our starting point is the worst-function construction of \citet[Appendix~A]{CB23}. It gives the same lower bound for a subclass of algorithms, including the subgradient method, whose queries remain in prescribed g-convex regions. We retain this geometric setup and carefully tune and augment their hard function to accommodate every deterministic first-order algorithm. The key new mechanism is an \emph{advance-or-stall scheme}: when a query lies in these regions, the oracle reveals at most one new information direction (\emph{advance}); when a query leaves them, we show that the oracle can answer without revealing \emph{any} new hidden direction (\emph{stall}). Unlike earlier curvature-dependent lower bounds, the proof uses no ball-packing argument.
\end{abstract}

\tableofcontents

\section{Introduction}\label{sec:intro}
We consider minimizing a geodesically convex (g-convex) function $f:\calM\to\RR$ on a Hadamard manifold $\calM$:
\begin{equation}\tag{P}\label{eq:optimization}
\min_{x\in\calM} f(x).
\end{equation}
When $\calM$ is Euclidean space $\RR^d$, \eqref{eq:optimization} corresponds to standard convex optimization.

Examples of g-convex optimization on curved spaces include computing intrinsic means and medians \citep{Karcher77,Fletcher09,bacak2014hadamard}; robust covariance estimation and subspace recovery \citep{auderset2005angular,wiesel2012gconvexity,wiesel2015gconvexity,ciobotaru2020geometrical,franksmoitra2020}; estimation for matrix normal models \citep{tang2021integrated,amendola2021,franks2021neartyler}; the quantum marginal problem \citep{tensorscaling2018}; and computing Brascamp--Lieb constants \citep{sravishnoibracsampliebconstant}. Many of these applications fall under the broader framework of scaling problems; see \citet{BFGOWW19}.

We study the first-order query complexity of g-convex optimization: how much information must an algorithm obtain before it can find an approximate minimizer? Hyperbolic space $\HH^d$---which has constant negative curvature---provides the simplest setting in which to identify the cost of departing from Euclidean geometry. Throughout, we therefore restrict our attention to $\calM=\HH^d$ and to one function class: Lipschitz, potentially nonsmooth, g-convex functions.\footnote{We expect our lower bound to extend to smooth g-convex optimization through a suitable smoothing argument, in the spirit of \citet{GN15} and \citet[Section~5]{CB23}. This extension is left to future work.} We now define the problem and its oracle model precisely.

\subsection{Problem class, oracles, and query complexity}\label{subsec:problem}
Let $\HH^d$ denote $d$-dimensional hyperbolic space with constant sectional curvature $-1$ (see \Cref{subsec:model}).\footnote{For constant curvature $-K<0$, replace $\zeta_r$ by $\zeta_{r\sqrt K}$ throughout the paper. The normalization to curvature $-1$ is only for convenience.} For each $d\ge2$, fix $\xref^d\in\HH^d$; we often write $\xref$ for $\xref^d$, suppressing the superscript. Write $\dist$ for Riemannian distance and $B(\xref,r)$ for the closed ball of radius $r$ centered at $\xref$. For $r,M>0$, the function class under consideration is
\begin{equation}\tag{fct class}\label{eq:class}
\calF_{r,M}^d
=\left\{f:\HH^d\to\RR:
\begin{array}{l}
f\text{ is globally g-convex and }M\text{-Lipschitz},\\
f\text{ has a global minimizer in }B(\xref,r)
\end{array}\right\}.
\end{equation}
Also let $\calF_{r,M}=\bigcup_{d\ge2}\calF_{r,M}^d$. Geodesic convexity, Lipschitz continuity, and subgradients are defined in \Cref{subsec:gconvex}. We write $f^*=\min f$. Given $0<\eps<1$, the \emph{dimension-independent Lipschitz g-convex optimization problem} is
\begin{equation}\tag{Lip}\label{eq:problem}
\text{given }f\in\calF_{r,M},\quad
\text{find }x\quad\text{such that}\quad
f(x)-f^*\le\eps Mr.
\end{equation}
\citet{CB23} call this problem P2. The ball $B(\xref,r)$ in \eqref{eq:class} specifies our \emph{initial uncertainty} about a global minimizer, not a restriction on where an algorithm may query.

A \emph{first-order oracle} for a g-convex function $f:\HH^d\to\RR$ is a map
\[
\calO_f(x)=(f(x),g(x)),\qquad g(x)\in\partial f(x).
\]
A query at $x$ thus returns the exact function value and a subgradient at $x$. A \emph{deterministic first-order algorithm} on $\HH^d$ is a sequence of maps
\[
\calA^d=\bigl(\calA_k^d:(\RR\times\T\HH^d)^k\to\HH^d\bigr)_{k\ge0},
\]
where $\T\HH^d$ denotes the tangent bundle. The maps may depend on the known parameters $d,r,M,\eps$ and $\xref^d$. Starting with the empty transcript $\mathcal D_0=\varnothing$, the interaction of $\calA^d$ with an oracle $\calO_f$ produces query points $x_0,x_1,\ldots$ recursively by
\[
x_k=\calA_k^d(\mathcal D_k),\qquad
(F_k,g_k)=\calO_f(x_k),\qquad
\mathcal D_{k+1}=\bigl(F_\ell,(x_\ell,g_\ell)\bigr)_{\ell=0}^{k}.
\]
Thus $\mathcal D_k$ records the information from the first $k$ queries. The algorithm may query anywhere in $\HH^d$, and computations between oracle calls are unrestricted.

For a deterministic first-order algorithm $\calA^d$, a function $f$, and a valid oracle $\calO_f$, we count queries up to and including the first one satisfying \eqref{eq:problem}:
\[
T_{\eps Mr}(\calA^d,f,\calO_f)
=\inf\{k+1:k\ge0,\ f(x_k)-f^*\le\eps Mr\},
\]
with $\inf\varnothing=+\infty$. Let $\calA=(\calA^d)_{d\ge2}$ be a family of such algorithms. The worst-case \emph{query complexity} of solving \eqref{eq:problem} is
\begin{equation}\tag{query complexity}\label{eq:query-complexity}
T_{\eps,r}
=\inf_{\calA}\ \sup_{d\ge2}\ \sup_{f\in\calF_{r,M}^d}\ \sup_{\calO_f}
T_{\eps Mr}(\calA^d,f,\calO_f),
\end{equation}
where the last supremum is over all valid first-order oracles for $f$. Scaling $f$ shows that $T_{\eps,r}$ is independent of $M$: it depends only on $\eps$ and $r$. In Euclidean convex optimization it is also independent of $r$, because rescaling the domain preserves the geometry. Crucially, hyperbolic space has no such scale invariance.

The effect of curvature on the complexity $T_{\eps,r}$ is therefore expressed through its dependence on $r$, and in particular through the quantity
\begin{equation}\tag{curvature factor}\label{eq:zeta}
\zeta_r=\frac{r}{\tanh r}.
\end{equation}
For small $r$, the geometry of $\HH^d$ inside $B(\xref,r)$ is approximately Euclidean, and indeed $\zeta_r\to1$ as $r\downarrow0$. For large $r$, curvature has a greater effect: $\zeta_r\sim r$ as $r\to\infty$.

The dimension-independent qualification is essential. In fixed dimension, cutting-plane methods can have \emph{logarithmic} dependence on the target accuracy; see \citet{Rusciano20,CMB23}. In contrast, our hard instances have dimension $d$ proportional to the query horizon.

\subsection{Previous bounds and the main result}\label{subsec:main-result}
The dimension-independent first-order complexity of Euclidean Lipschitz convex optimization is $\Theta(\eps^{-2})$, and the subgradient method is optimal \citep{NY83,Nesterov04}. For g-convex optimization, \citet{ZS16} give an $O(\zeta_r/\eps^2)$ upper bound for the Riemannian subgradient method, which in turn implies
\begin{equation}\tag{subgradient bound}\label{eq:upper-complexity}
T_{\eps,r}=O\!\left(\frac{\zeta_r}{\eps^2}\right).
\end{equation}
Is this bound tight? A better algorithm might instead achieve, for example, $O(\zeta_r+\eps^{-2})$, paying \emph{additively} rather than \emph{multiplicatively} for curvature and accuracy.

\citet{HM21} proved an $\widetilde\Omega(\zeta_r)$ lower bound for smooth strongly g-convex optimization on the hyperbolic plane, assuming \emph{noisy} first-order oracles.\footnote{The notation $\widetilde\Omega$ suppresses logarithmic factors.} Building on that work, \citet{CB22} showed that this lower bound persists for exact first-order oracles, as considered here. Building in part on those results, \citet{CB23} proved for \eqref{eq:problem} the lower bound
\[
T_{\eps,r}=\widetilde\Omega\!\left(\zeta_r+\frac{1}{\zeta_r^2\eps^2}\right)
\]
and \emph{conjectured} that the multiplicative upper bound \eqref{eq:upper-complexity} is optimal.

As evidence for their conjecture, \citet[Appendix~A]{CB23} prove an $\Omega(\zeta_r/\eps^2)$ lower bound for a \emph{subclass} of algorithms that includes a Polyak subgradient method. The present paper builds substantially on that construction: we use its geometric setup (\Cref{sec:geometry}) and tune and augment its hard-function ansatz (\Cref{sec:construction}). The main ideas and their relation to the earlier proof are explained in \Cref{subsec:overview}.

\Needspace{4\baselineskip}
Our main result is the following theorem.

\begin{theorem}[Tight lower bound]\label{thm:main}
Let $c=2^{-16}$ and $\eps_0=2^{-9}$. For $r,M>0$ and $0<\eps\le\eps_0$, set
\begin{equation}\label{eq:main-budget}
T=\left\lfloor c\frac{\zeta_r}{\eps^2}\right\rfloor,
\qquad d=2T.
\end{equation}
For every deterministic first-order algorithm $\calA^d$ on $\HH^d$, there exist $f\in\calF_{r,M}^d$ and a valid first-order oracle $\calO_f$ such that running $\calA^d$ on $\calO_f$ produces queries $x_k$ satisfying
\begin{equation}\label{eq:main-gap}
f(x_k)-f^*\ge2\eps Mr
\qquad\text{for every }k=0,\ldots,T-1.
\end{equation}
The function $f$ has a unique minimizer $x^*\in\partial B(\xref,r)$. The conclusion holds even if $f^*=0$ is disclosed to the algorithm in advance.
\end{theorem}

Combining \Cref{thm:main} with \eqref{eq:upper-complexity} gives
\begin{equation}\label{eq:complexity}
T_{\eps,r}=\Theta\!\left(\frac{\zeta_r}{\eps^2}\right)
\qquad(r>0,\ 0<\eps\le\eps_0).
\end{equation}
This answers the conjecture of \citet{CB23} in the affirmative for the Lipschitz problem \eqref{eq:problem}. In particular, at any sufficiently small fixed relative accuracy, the construction gives an $\Omega(\zeta_r)$ lower bound \emph{without logarithmic losses} and \emph{without ball packing}.

\begin{remark}[Subgradient selection and smoothing]\label{rem:smoothing}
\Cref{thm:main} is stated for a nonsmooth hard function and an exact subgradient oracle. A $C^1$ refinement with unique subgradients can be obtained by slightly separating successive ramp families (\Cref{subsec:ansatz}) and applying sufficiently local Riemannian Moreau regularization. The regularity and localization properties of this regularization are given by \citet[Lemma~10 and Appendix~G]{CB23}, following \citet{AF06}. We omit that additional construction.
\end{remark}

\paragraph{Concurrent work.}
In concurrent work posted to arXiv one week earlier, \citet{WuEtAl26} independently obtained the same tight $\Theta(\zeta_r/\eps^2)$ complexity. Their proof shares three central ingredients with \citet[Appendix~A]{CB23}, and hence also with the proof developed here: \textbf{(a)} the nested hyperbolic geometry recalled in \Cref{sec:geometry}, which is crucial for producing the $\zeta_r$ factor (see \Cref{subsec:overview}); \textbf{(b)} the use of distance to a hidden minimizer as the base function (see \Cref{subsec:overview,prop:sharp-support}); and \textbf{(c)} a distance-based correction mechanism closely tied to hyperbolic halfspaces.

Nevertheless, the hard functions constructed by \citeauthor{WuEtAl26} and in the present paper, as well as the mechanisms for handling arbitrary queries, are quite different. Our hard-function construction remains close to the ansatz of \citet{CB23}, augmented with the hyperbolic reflections suggested by \citet[Chapter~5, Remark~5.1.17]{CriscitielloThesis25}:
\[
f(x)=\dist(x,x^*)+\alpha\max_{j,\sigma}\dist(x,L_{j,\sigma}),
\]
as in \eqref{eq:core}. \citeauthor{WuEtAl26} instead grow a ``convex carrier'' $K$ and use the hard function
\[
f(x)=\min_{y\in K}
\left\{\mu\,\dist(y,x^*)+\dist(x,y)\right\}.
\]
\citeauthor{WuEtAl26} construct a resisting oracle directly for arbitrary queries, whereas we first establish the lower bound under the geodesic-span assumption (see \Cref{subsec:A1proof}) and then remove this assumption through the black-box reduction of \Cref{sec:unrestricted}.

\subsection{Construction and geometric ideas}\label{subsec:overview}
In this subsection, we describe the key ideas behind the lower bound construction.

\paragraph{First-order supports and g-convex building blocks.}
Classical lower bounds for Lipschitz convex optimization use maxima of affine functions \citep{NY83,Nesterov04}. This is natural: a finite convex function is the supremum of its affine supports $f(y)+\inner{g}{x-y}$, where $g\in\partial f(y)$. The g-convex analogue is (see \Cref{subsec:gconvex})
\begin{equation}\tag{support inequality}\label{eq:intro-support}
f(x)\ge f(y)+\inner{g}{\log_y(x)},\qquad g\in\partial f(y).
\end{equation}
On hyperbolic space, however, the support $x\mapsto\inner{g}{\log_y(x)}$ is \emph{not} g-convex unless $g=0$. One therefore cannot simply follow the Euclidean construction and obtain a hard g-convex function by taking maxima of these supports.

Nevertheless, \eqref{eq:intro-support} is sharp: for prescribed $f(y)$, $g\in\partial f(y)$, and a target point $x$, there is a g-convex function satisfying $f(x)=f(y)+\inner{g}{\log_y(x)}$; see \Cref{prop:sharp-support}. On hyperbolic space, the realizing function is a nonnegative combination of distance to a point and distance to a halfspace, plus a normalizing constant.

\paragraph{A hard function from distance corrections.}
This sharpness of \eqref{eq:intro-support} motivates the hard-function ansatz used by \citet[Appendix~A]{CB23},
\[
f(x)=\dist(x,x^*)+\alpha\max_j\dist(x,L_j),
\]
where the $L_j$ are hyperbolic halfspaces (see \Cref{subsec:model}); the precise ansatz is developed in \Cref{subsec:ansatz}. We call $x\mapsto\dist(x,L_j)$ a \emph{ramp}: it vanishes inside $L_j$ and increases at unit rate along outward normal geodesics. Without the ramps, the gradient of $f$ points directly away from $x^*$, revealing the direction toward the minimizer. The ramps are designed to deflect this gradient into an ``information submanifold''. Specifically, following the role of linear spans in Euclidean lower bounds \citep{Nesterov04}, we construct nested totally geodesic submanifolds $I_0\subseteq I_1\subseteq\cdots$ as in \Cref{sec:geometry}. At a query $x_k\in I_k$, we seek a subgradient of $f$ that is tangent to $I_{k+1}$, so that the oracle exposes at most one new direction. \citet[Appendix~A]{CB23} choose halfspaces $L_j$ that achieve this, \emph{provided} each query belongs to all previously introduced halfspaces, that is, $x_k\in L_j$ for $j<k$; they call this requirement assumption A2.

\paragraph{Reflections, tuning, and advance-or-stall.}
An unrestricted algorithm may query outside a halfspace $L_j$, violating A2. To accommodate such queries, we replace each halfspace by a set of reflected copies $L_{j,\sigma}$ and consider
\begin{equation}\tag{hard fct}\label{eq:intro-ansatz}
f(x)=\dist(x,x^*)+\alpha\max_j\max_\sigma\dist(x,L_{j,\sigma}),
\end{equation}
where $\sigma$ indexes the reflection isometries described in \Cref{subsec:ansatz}.\footnote{Chapter~5, Remark~5.1.17 of \citeauthor{CriscitielloThesis25}'s thesis \citeyearpar{CriscitielloThesis25} proposes using reflections in this way to handle algorithms that do not satisfy A2.} At a query $x_k\in I_k$, a collection of reflected ramps coincide in value while their subgradients have different components normal to $I_k$. After multiplication by $\alpha$, any convex combination of these active subgradients is a subgradient of the right-hand correction term in \eqref{eq:intro-ansatz}. The idea is to select a convex combination whose contribution to that term deflects the gradient of $\dist(\cdot,x^*)$ into the desired information submanifold.

If $x_k$ lies inside every halfspace, all ramps vanish. We then use the ``fresh'' family of ramps $\dist(\cdot,L_{k,\sigma})$ to obtain a subgradient of the hard function tangent to $I_{k+1}$. We call this an \emph{advance}. If instead $x_k$ lies outside some halfspace, a positive ramp $\dist(\cdot,L_{j,\sigma})$ from an earlier stage $j<k$ attains the maximum. We show that its reflected family provides a subgradient of the hard function tangent to $I_k$, revealing no new hidden direction. We call this a \emph{stall}; both cases are made precise in \Cref{lem:stall-advance}.

Reflections alone do not give the desired $\Omega(\zeta_r/\eps^2)$ bound with the original halfspaces proposed by Criscitiello and Boumal. The halfspaces must be carefully tuned for a \emph{stall} to work even when the active ramp is very old. Unlike the original halfspaces, whose boundaries contain $x^*$, the new halfspaces we propose contain $x^*$ \emph{strictly in their interiors}. In the stall case, the boundary of a violated halfspace therefore separates the query from $x^*$ by a definite interior margin (\Cref{lem:margin}). This separation is crucial; see \Cref{lem:oldramp,lem:completion}.

\paragraph{Why hyperbolic right triangles give the multiplicative bound.}
The location of $x^*$ and the information submanifolds $I_k$ are chosen exactly as in \citet[Appendix~A]{CB23}. Specifically, $x^*$ lies on the boundary of $B(\xref,r)$, and the construction has the following property. If $y_k$ is the metric projection of $x^*$ onto $I_k$ and $r_k=\dist(y_k,x^*)$, then the angle at $y_k$ between the direction to $x^*$ and the next information direction, which generates $I_{k+1}$ from $I_k$, is a fixed angle $\theta$ close to $\pi/2$.\footnote{This nested right-triangle construction was motivated by the geometric interpretation of the Polyak subgradient method; see \citet[Appendix~A.3]{CB23}. For the hard function constructed there, the points $y_k$ are exactly the iterates generated by the Polyak subgradient method. See the end of \Cref{sec:geometry} for this connection in more detail and \Cref{subsec:ansatz} for the related hard-function construction.} Hyperbolic trigonometry (see \Cref{subsec:triangles,sec:geometry}) applied to the right triangle $(y_k,y_{k+1},x^*)$ gives
\begin{equation}\tag{radius recurrence}\label{eq:intro-radii}
\sinh r_{k+1}=\sin\theta\,\sinh r_k,
\qquad \sinh r_k=(\sin\theta)^k\sinh r.
\end{equation}
For small $\cos\theta$, one step satisfies $r_{k+1}^2\approx r_k^2(1-\cos^2\theta/\zeta_{r_k})$. In Euclidean space, the corresponding identity is $r_{k+1}^2=r_k^2(1-\cos^2\theta)$. This is where the additional factor $\zeta_r$ enters.

Iterating \eqref{eq:intro-radii} gives $r_k\ge r/2$ for $k\le\zeta_r/(2\cos^2\theta)$ \citep[Lemma~24]{CB23}. As long as query $x_k$ remains in $I_k$, \eqref{eq:intro-ansatz} therefore has gap at least $r/2$ at that query. Its Lipschitz constant is $1+\alpha$, with $\alpha$ of order $\sec\theta$, so normalization to Lipschitz constant $M$ leaves a gap of order $Mr\cos\theta$. Taking $\cos\theta$ proportional to $\eps$ gives the lower bound, provided each new query lies in the geodesic span of the information already received (\Cref{subsec:A1proof}). This is the natural counterpart of the linear-span assumption in Euclidean lower bounds.

\paragraph{Removing the geodesic-span assumption.}
Starting from a $T$-query lower bound for geodesic-span algorithms on $\HH^d$, \Cref{sec:unrestricted} transfers it to arbitrary deterministic algorithms on $\HH^{d+T}$. The argument follows the Euclidean hidden-subspace reduction \citep{CDHS20,Li26}, with an important modification. We place a hard function $f:\HH^d\to\RR$ on a totally geodesic copy $\mathcal N\subseteq\HH^D$, where $D=d+T$, and seek to extend it to a function $F:\HH^D\to\RR$. The embedding $\mathcal N$ is chosen adversarially so that the queries of an unrestricted algorithm for $F$ correspond, after metric projection onto $\mathcal N$, to those of a geodesic-span algorithm for $f$.

Identifying $f$ with its copy on $\mathcal N$, the natural extension is $x\mapsto f(\Proj_{\mathcal N}(x))$, where $\Proj_{\mathcal N}$ denotes metric projection onto $\mathcal N$. Unlike the Euclidean construction, however, this extension is not g-convex (unless $f$ is constant). We repair it by setting
\[
F(x)=f(\Proj_{\mathcal N}(x))+M\log\cosh\dist(x,\mathcal N).
\]
\Cref{lem:lifting} shows that $F$ is g-convex and $M$-Lipschitz, and that one oracle query to $f$ suffices to answer one oracle query to $F$. \Cref{thm:blackbox} constructs the embedding adversarially and proves this black-box reduction from geodesic-span lower bounds to lower bounds for unrestricted deterministic algorithms.

The right-triangle geometry, the sharpness of first-order supports, and the original distance-correction ansatz are all due to \citet{CB23}. The present work supplies the advance-or-stall argument, the choice and analysis of the tuned halfspaces, and the black-box reduction that removes the geodesic-span assumption.

\paragraph{Organization.} \Cref{sec:background} reviews the necessary background and identifies several ingredients motivating the lower bound. \Cref{sec:geometry} recalls the inductive geometric construction of Criscitiello and Boumal, and \Cref{sec:construction} contains the main geometric argument, proving the lower bound under a geodesic-span assumption. \Cref{sec:unrestricted} removes that assumption by the black-box reduction and completes the proof of \Cref{thm:main}.

The lower bound construction is somewhat intricate. On a first reading, it may help to imagine the corresponding construction in Euclidean space, as depicted in the figures.

\section{Geometric background}\label{sec:background}
In this section, we introduce the necessary background on hyperbolic space and geodesic convexity. We also highlight two guiding facts---boxed below---that motivate the geometry behind our lower bound construction. These facts, and most of our notation, come from \citet{CB23}.

\subsection[Hadamard manifolds, hyperbolic space, and halfspaces]{Hadamard manifolds, hyperbolic space, totally geodesic submanifolds, and halfspaces}\label{subsec:model}
Let $\calM$ be a \emph{Hadamard manifold}: a complete, simply connected Riemannian manifold with nonpositive sectional curvature; see \citet{BH99} and \citet[Chapter~12]{Lee18}. For $x\in\calM$, its tangent space is $\T_x\calM$. We denote the Riemannian metric by $\inner{\cdot}{\cdot}$ and its induced norm by $\norm{v}=\sqrt{\inner{v}{v}}$. Riemannian distance is denoted by $\dist$, and $B(x,r)$ is the closed ball of radius $r$ centered at $x$. The exponential map $\exp_x:\T_x\calM\to\calM$ is a global diffeomorphism, with inverse $\log_x$. Every pair of points $x,y\in\calM$ is joined by a unique minimizing geodesic. We write $P_{x\to y}:\T_x\calM\to\T_y\calM$ for parallel transport along that geodesic. The boundary and interior of a set $D\subseteq\calM$ are denoted by $\partial D$ and $\interior D$, respectively.

Examples of Hadamard manifolds include Euclidean space $\RR^d$, $d$-dimensional hyperbolic space $\HH^d$ (see the hyperboloid model below), and positive-definite matrices with the affine-invariant metric \citep[Chapter~6]{Bhatia}.

A connected, complete Riemannian submanifold $S\subseteq\calM$ is \emph{totally geodesic} if every geodesic of $S$ is also a geodesic of $\calM$. Throughout, ``totally geodesic submanifold'' includes connectedness and completeness. These submanifolds are abundant in hyperbolic space, unlike on general Hadamard manifolds \citep[Section~11.1]{Chen00}. If $x,y$ are elements of a totally geodesic submanifold $S\subseteq\calM$, then $\log_x(y)\in\T_xS$; if $v\in\T_xS$, then $\exp_x(tv)\in S$ for every $t\in\RR$. Intrinsic and ambient distances on $S$ agree, and parallel transport along curves in $S$ preserves its tangent and normal spaces \citep[Lemmas~19--20]{CB23}.

\paragraph{The hyperboloid model.}
To make our constructions as concrete as possible, we use the hyperboloid model of hyperbolic space; see \citet{BH99} and \citet[Chapter~3]{Ratcliffe19}. Coordinates in $\RR^{d+1}$ are indexed by $0,1,\ldots,d$, with standard basis $e_0,e_1,\ldots,e_d$. Define the Lorentz inner product by
\[
\innerL{x}{y}=-x_0y_0+\sum_{i=1}^d x_i y_i.
\]
The subscript $\mathrm L$ distinguishes this ambient bilinear form from the intrinsic Riemannian metric. The model is
\begin{equation}\tag{hyperboloid model}\label{eq:model}
\HH^d=\{x\in\RR^{d+1}:\innerL{x}{x}=-1,\ x_0>0\}.
\end{equation}
Any linear map preserving the Lorentz form and the sheet $x_0>0$ restricts to an isometry of $\HH^d$. We shall use coordinate reflections of this kind in \Cref{subsec:ansatz}.

Its tangent spaces are
\[
\T_x\HH^d=x^\perp=\{v\in\RR^{d+1}:\innerL{v}{x}=0\},
\]
where orthogonality is with respect to the Lorentz form. The restriction of that form $\innerL{\cdot}{\cdot}$ to $\T_x\HH^d$ is positive definite and gives the Riemannian metric. For $x,y\in\HH^d$ and $v\in\T_x\HH^d$,
\begin{align}
\cosh\dist(x,y)&=-\innerL{x}{y},\tag{hyperboloid dist}\label{eq:hyperboloid-distance}\\
\exp_x(v)&=\cosh\norm{v}\,x+\frac{\sinh\norm{v}}{\norm{v}}v,\tag{hyperboloid exp}\label{eq:exp}\\
\log_x(y)&=\frac{\dist(x,y)}{\sinh\dist(x,y)}
\bigl(y-\cosh\dist(x,y)\,x\bigr).\tag{hyperboloid log}\label{eq:log}
\end{align}
The last two formulas are interpreted continuously at $v=0$ and $y=x$.

A $k$-dimensional totally geodesic submanifold of $\HH^d$ has the form $\HH^d\cap E$ for a $(k+1)$-dimensional linear subspace $E$ with nonempty intersection, and is isometric to $\HH^k$. Conversely, every nonempty intersection with a linear subspace is totally geodesic \citep[Lemma~21]{CB23}. Thus, for points and tangent vectors $(x_i,v_i)\in\T\HH^d$, define
\begin{equation}\tag{gspan}\label{eq:gspan}
\gspan(x_1,\ldots,x_m,v_1,\ldots,v_m)
=\HH^d\cap\spanof(x_1,\ldots,x_m,v_1,\ldots,v_m).
\end{equation}
Here $\spanof$ denotes the ordinary real linear span in $\RR^{d+1}$. Intrinsically, \eqref{eq:gspan} is the smallest totally geodesic submanifold that contains each $x_i$ and is tangent to each $v_i$ \citep[Proposition~23]{CB23}.

We also allow totally geodesic submanifolds as arguments of $\gspan$, interpreting them as collections of points. In particular, if $S=\HH^d\cap E$ is a totally geodesic submanifold, then, for $x\in\HH^d$ and $v\in\T_x\HH^d$,
\[
\gspan(S,x,v)=\HH^d\cap\spanof(E,x,v).
\]

\paragraph{Metric projection onto a totally geodesic submanifold.}
The following formulas are needed only in \Cref{sec:unrestricted}. We state them for ambient dimension $D$, to distinguish that space from the lower-dimensional space used there. For a nonempty totally geodesic submanifold $\mathcal N\subseteq\HH^D$, write $\Proj_{\mathcal N}$ for metric projection onto $\mathcal N$: for each $x\in\HH^D$, $\Proj_{\mathcal N}(x)$ is the nearest point on $\mathcal N$ in Riemannian distance.

\begin{lemma}[Hyperbolic projection]\label{lem:hyperbolic-projection}
Let $\mathcal N=\HH^D\cap V$ be nonempty, and let $\Pi_V$ be Lorentz-orthogonal linear projection onto the subspace $V\subseteq\RR^{D+1}$. For $x\in\HH^D$, define $\lambda=\sqrt{-\innerL{\Pi_Vx}{\Pi_Vx}}$. Then
\begin{equation}\tag{metric projection}\label{eq:metric-projection}
\Proj_{\mathcal N}(x)=\frac{\Pi_Vx}{\lambda},
\qquad \lambda=\cosh\dist(x,\mathcal N)\ge1.
\end{equation}
In particular, this projection is smooth. For every $z\in\mathcal N$,
\begin{equation}\tag{projection triangle}\label{eq:general-projection-triangle}
\cosh\dist(x,z)
=\cosh\dist(x,\mathcal N)\,
 \cosh\dist(\Proj_{\mathcal N}(x),z).
\end{equation}
The metric projection of an ambient geodesic onto $\mathcal N$ is contained in a geodesic of $\mathcal N$, or consists of a single point.
\end{lemma}
\begin{proof}
Choose $z_0\in\mathcal N$. Since $V^\perp\subseteq z_0^\perp=\T_{z_0}\HH^D$, the Lorentz form is positive definite on $V^\perp$: its restriction there is part of the Riemannian metric at $z_0$. For $x\in\HH^D$, write $x=\Pi_Vx+w$ with $w\in V^\perp$. Orthogonality and $\innerL{x}{x}=-1$ give
\[
\lambda^2=-\innerL{\Pi_Vx}{\Pi_Vx}=1+\innerL{w}{w}\ge1.
\]
Thus $p=\Pi_Vx/\lambda$ belongs to $V$ and satisfies $\innerL{p}{p}=-1$. It is on the positive sheet: indeed,
\[
-\innerL{p}{z_0}=-\innerL{x}{z_0}/\lambda
=\cosh\dist(x,z_0)/\lambda>0.
\]
Hence $p\in\mathcal N$.

Set $\delta=\arccosh \lambda$. If $\delta>0$, then $n=w/\sinh\delta$ is a unit vector in $V^\perp$, so it is tangent at $p$ and normal to $\mathcal N$. Formula \eqref{eq:exp} gives
\[
x=\cosh\delta\,p+\sinh\delta\,n=\exp_p(\delta n).
\]
For $z\in\mathcal N$, the geodesic from $p$ to $z$ lies in $\mathcal N$. The triangle $(x,p,z)$ is therefore right at $p$, and \eqref{eq:right-triangle} yields
\[
\cosh\dist(x,z)=\cosh\delta\,\cosh\dist(p,z).
\]
This identity also holds when $\delta=0$, since then $w=0$ and $x=p$. Its right-hand side is minimized uniquely at $z=p$, proving \eqref{eq:metric-projection} and \eqref{eq:general-projection-triangle}. The formula for $p$ is smooth because $\Pi_V$ is linear and $\lambda\ge1$.

Finally, an ambient geodesic lies in a two-dimensional linear subspace $E\subseteq\RR^{D+1}$. Its metric projections lie in $\HH^D\cap\Pi_V(E)$. This intersection is nonempty and $\dim\Pi_V(E)\le2$, so the characterization preceding \eqref{eq:gspan} makes it a geodesic or a point.
\end{proof}

\paragraph{Halfspaces and their distance functions.}
Let $W\ne0$ be tangent at $z\in\HH^d$. The \emph{hyperbolic halfspace} associated with $W$ is
\begin{equation}\tag{halfspace}\label{eq:halfspace}
\begin{split}
L&=\exp_z\{v\in\T_z\HH^d:\inner{W}{v}\ge0\}\\
 &=\{x\in\HH^d:\innerL{W}{x}\ge0\}.
\end{split}
\end{equation}
This equivalence follows from \eqref{eq:exp} and $\innerL{W}{z}=0$. Its boundary $\partial L=\HH^d\cap W^\perp$ is a closed, totally geodesic submanifold of dimension $d-1$, called a \emph{hyperbolic hyperplane}.

Put $\nu=W/\norm{W}$, the inward unit normal at $z$. Regard $\nu$ also as a constant ambient vector. The signed distance to $\partial L$, positive on the $L$ side, and the distance to $L$ are
\begin{equation}\tag{ramp distance}\label{eq:signed-distance}
s_L(x)=\arcsinh\innerL{\nu}{x},
\qquad \dist(x,L)=\positive{-s_L(x)}.
\end{equation}
Here $[t]_+=\max\{t,0\}$. To derive the first identity, let $p=\Proj_{\partial L}(x)$ be the metric projection of $x$ onto $\partial L$. The geodesic from $p$ to $x$ is perpendicular to $\partial L$. If $s$ is its signed length, \eqref{eq:exp} gives
\[
x=\cosh s\,p+\sinh s\,\nu.
\]
Taking the Lorentz inner product with $\nu$ yields $\innerL{\nu}{x}=\sinh s$, proving \eqref{eq:signed-distance}.

Let $y\notin L$ and $\delta=\dist(y,L)$. In the hyperboloid model, the gradient of $\dist(\cdot,L)$ at $y$ is
\begin{equation}\tag{ramp slope}\label{eq:ramp-gradient}
\innerL{\nu}{y}=-\sinh\delta,
\qquad
\nabla\dist(\cdot,L)(y)
=-\frac{\nu-\sinh\delta\,y}{\cosh\delta}.
\end{equation}
This follows by differentiating \eqref{eq:signed-distance}. At $y\in\partial L$, the corresponding subdifferential is $\{-t\nu:0\le t\le1\}$; subdifferentials are defined in \Cref{subsec:gconvex}.

\subsection{Hyperbolic trigonometry}\label{subsec:triangles}
Consider a right triangle with vertices $y_0,y_1,x^*\in\HH^d$, right angle at $y_1$, and angle $\theta\in(0,\pi/2)$ at $y_0$. Write
\[
r_0=\dist(y_0,x^*),\qquad r_1=\dist(y_1,x^*),\qquad
\Delta_0=\dist(y_0,y_1).
\]
Hyperbolic right-triangle identities \citep[Exercise~3.5.4]{Ratcliffe19} give
\begin{equation}\tag{hyperbolic trig}\label{eq:right-triangle}
\sin\theta=\frac{\sinh r_1}{\sinh r_0},\qquad
\cos\theta=\frac{\tanh\Delta_0}{\tanh r_0},\qquad
\cosh r_0=\cosh r_1\cosh\Delta_0.
\end{equation}
For fixed $r_0>0$, expand the first identity as $\cos\theta\to0$. The derivative of $\arcsinh(\sinh r_0\sqrt{1-t})$ at $t=0$ is $-\tanh r_0/2$, so
\begin{align}
r_1&=r_0-\tfrac12\tanh r_0\cos^2\theta+O_{r_0}(\cos^4\theta),\notag\\
r_1^2&=r_0^2\left(1-\frac{\cos^2\theta}{\zeta_{r_0}}\right)
+O_{r_0}(\cos^4\theta).\label{eq:triangle-expansion}
\end{align}
In Euclidean space, $r_1^2=r_0^2(1-\cos^2\theta)$ exactly. We conclude that the side opposite $\theta$ is longer in hyperbolic space. The appearance of $\zeta_{r_0}$ in this comparison is precisely what produces the factor $\zeta_r$ in our $\Omega(\zeta_r/\eps^2)$ lower bound; see \Cref{lem:geometry} and the proof of \Cref{thm:A1}, where that lemma is used.

\insight{The hyperbolic residual is longer}{For small $\cos\theta$, a right triangle with hypotenuse $r_0$ has
\[
\begin{array}{ll}
\text{hyperbolic:}&r_1^2\approx r_0^2\bigl(1-\cos^2\theta/\zeta_{r_0}\bigr),\\[2pt]
\text{Euclidean:}&r_1^2=r_0^2(1-\cos^2\theta).
\end{array}
\]
}

\subsection{Geodesic convexity and sharp first-order supports}\label{subsec:gconvex}
A subset of a Hadamard manifold is \emph{geodesically convex} (\emph{g-convex}) if it contains the geodesic segment between each pair of its points. In $\HH^d$, balls $B(x,r)$, hyperbolic halfspaces \eqref{eq:halfspace}, and totally geodesic submanifolds, including hyperbolic hyperplanes, are g-convex \citep[Lemmas~1 and 19]{CB23}. In contrast, the exponential image of a tangent halfspace need not be g-convex on a general Hadamard manifold.

A function $f:\calM\to\RR$ is \emph{g-convex} if its restriction to every constant-speed geodesic is convex. We say that $f$ is \emph{$M$-Lipschitz} if $|f(x)-f(y)|\le M\dist(x,y)$ for all $x,y\in\calM$. A vector $g\in\T_y\calM$ is a \emph{subgradient} of $f$ at $y$ if
\begin{equation}\tag{first-order support}\label{eq:subgradient}
f(x)\ge f(y)+\inner{g}{\log_y(x)}\qquad\forall x\in\calM.
\end{equation}
The \emph{subdifferential} at $y$, denoted $\partial f(y)$, is the set of such tangent vectors. For finite g-convex functions on a Hadamard manifold, the subdifferential is always nonempty. If $f$ is globally $M$-Lipschitz, every subgradient has norm at most $M$. References for geodesic convexity include \citet{Udriste94} and \citet[Chapter~11]{Boumal23}.

Distance to a nonempty closed g-convex set is g-convex and $1$-Lipschitz on a Hadamard manifold \citep[Chapter~3, Section~11]{Udriste94}. Hence our ramps and distance-to-point functions are g-convex. A maximum of g-convex functions is g-convex. Let $P=\max_i p_i$ for finitely many g-convex functions $p_i$, and let $A=\{i:p_i(y)=P(y)\}$ be the set of active indices at $y$. If $q_i\in\partial p_i(y)$ for $i\in A$, then
\begin{equation}\tag{subdiff of max}\label{eq:active-combination}
\sum_{i\in A}\omega_i q_i\in\partial P(y)
\quad\text{for}\quad
\omega_i\ge0,\qquad\sum_{i\in A}\omega_i=1.
\end{equation}
Indeed, average the inequalities $P(x)\ge p_i(x)\ge P(y)+\inner{q_i}{\log_y(x)}$.

\paragraph{Why affine supports do not transfer.}
In Euclidean space, affine functions are the building blocks of convexity: every finite convex function is the supremum of its affine supports, and every finite maximum of affine functions is convex. For g-convexity, the lower bound \eqref{eq:subgradient} remains valid, but its right-hand side is not generally g-convex as a function of $x$. Taking maxima of first-order supports therefore does not automatically yield a g-convex function. This is the main obstacle to transferring Euclidean worst-function constructions to Hadamard manifolds; see \citet[Section~8 and Appendix~C]{CB23}. Nevertheless, these lower support functions are pointwise sharp.

\begin{proposition}[Pointwise sharpness of a first-order support, {\citealp[Proposition~35]{CB23}}]\label{prop:sharp-support}
Let $d\ge2$, $F\in\RR$, $y\in\HH^d$, and $g\in\T_y\HH^d$. For every $x\in\HH^d$,
\[
\min\{f(x):f:\HH^d\to\RR\text{ is g-convex},\ f(y)=F,\ g\in\partial f(y)\}
=F+\inner{g}{\log_y(x)}.
\]
The minimum can be attained by a nonnegative combination of distance to a point and distance to a hyperbolic halfspace, plus a constant. We emphasize that the minimizing function must, in general, depend on $x$.
\end{proposition}
\citet[Proposition~35]{CB23} prove the pointwise sharpness statement on every Hadamard manifold. We repeat their proof with a minor adaptation, because it motivates the hard-function ansatz used below.
\begin{proof}
The lower bound follows from \eqref{eq:subgradient}. First suppose $x\ne y$, and decompose
\[
g_\parallel
=\frac{\inner{g}{\log_y(x)}}{\dist(x,y)^2}\log_y(x),
\qquad g_\perp=g-g_\parallel.
\]
Choose
\[
x'=\begin{cases}
x,&\inner{g}{\log_y(x)}\le0,\\
\exp_y(-\log_y(x)),&\inner{g}{\log_y(x)}>0.
\end{cases}
\]
When $g_\perp\ne0$, define the hyperbolic halfspace
\[
L=\exp_y\{v\in\T_y\HH^d:\inner{g_\perp}{v}\le0\}.
\]
Its boundary contains the complete geodesic through $y$ and $x$, because $g_\perp$ is orthogonal to $\log_y(x)$. When $g_\perp=0$, choose any hyperbolic halfspace whose boundary contains that geodesic. In either case, $y,x,x'\in\partial L$. Define
\begin{equation}\tag{sharp support}\label{eq:sharp-support-function}
f(z)=F+\norm{g_\parallel}\bigl(\dist(z,x')-\dist(y,x')\bigr)
+\norm{g_\perp}\dist(z,L).
\end{equation}

Both distance terms are g-convex, and $f(y)=F$. The point $x'$ is distinct from $y$, so the first distance term is differentiable at $y$, with gradient
\[
-\frac{\norm{g_\parallel}}{\dist(y,x')}\log_y(x')=g_\parallel.
\]
This formula also holds when $g_\parallel=0$, since that term is constant. If $g_\perp\ne0$, the outward unit normal to $L$ at $y$ is $g_\perp/\norm{g_\perp}$. The boundary subgradient formula following \eqref{eq:ramp-gradient} therefore gives
\[
g_\perp\in\partial\bigl(\norm{g_\perp}\dist(\cdot,L)\bigr)(y).
\]
Adding the two supporting inequalities proves $g=g_\parallel+g_\perp\in\partial f(y)$.

At the target $x$, the ramp vanishes. If $x'=x$, \eqref{eq:sharp-support-function} gives $f(x)=F-\norm{g_\parallel}\dist(x,y)$. Otherwise $y$ is the midpoint of the geodesic from $x$ to $x'$, so $\dist(x,x')=2\dist(x,y)$ and $f(x)=F+\norm{g_\parallel}\dist(x,y)$. By the choice of $x'$, both cases give $f(x)=F+\inner{g}{\log_y(x)}$.

For $x=y$, the function $z\mapsto F+\norm{g}\dist(z,y)$ has value $F$ at $y$ and subgradient $g$ there, so the same conclusion holds.
\end{proof}

The realizing function \eqref{eq:sharp-support-function} is exactly a weighted distance-to-point function plus a distance-to-halfspace correction. The distance-to-point term supplies the component of the prescribed subgradient along the geodesic to the target; the ramp supplies its perpendicular component without changing the value at the target. This is the mechanism underlying the hard-function ansatz of \citet[Appendix~A and Remark~36]{CB23}, and it also motivates the hard function used here. We summarize it as follows.

\insight{Distance corrections attain first-order supports}{At any prescribed target, the first-order lower bound \eqref{eq:subgradient} is attained by a nonnegative combination of distance to a point and distance to a hyperbolic halfspace, plus a constant; see \eqref{eq:sharp-support-function}.}

\subsection{Busemann functions}\label{subsec:busemann}
This subsection is used only in \Cref{sec:unrestricted}, when removing the geodesic-span assumption. For a unit-speed geodesic ray $\gamma$ in a Hadamard manifold, its \emph{Busemann function} is
\[
b_\gamma(x)=\lim_{t\to\infty}\bigl(\dist(x,\gamma(t))-t\bigr).
\]
Busemann functions are g-convex and $1$-Lipschitz \citep[Chapter~II.8]{BH99}; on hyperbolic space they are also smooth, as the formula below shows. In Euclidean space, they are exactly the affine functions with unit-norm gradient: the ray $\gamma(t)=y+tv$, with $\norm v=1$, gives $b_\gamma(x)=-\inner{v}{x-y}$.

Busemann functions thus generalize affine functions and are fundamental in the geometry of Hadamard manifolds and noncompact symmetric spaces; see \citet[Chapters~II.8--9]{BH99}. They also enter applications of g-convex optimization: weighted Busemann functions describe stability at infinity \citep{KLM09}, while related asymptotic and dual constructions apply to scaling problems \citep{Hirai24}.

\paragraph{The hyperboloid formula.}
A nonzero vector $\xi\in\RR^{D+1}$ is \emph{future-pointing null} if $\xi_0>0$ and $\innerL{\xi}{\xi}=0$. Its positive multiples represent the same ideal endpoint of $\HH^D$, that is, a point on its boundary at infinity \citep[Chapters~II.8--9]{BH99}. Up to an additive constant, the associated Busemann function is
\begin{equation}\tag{Busemann function}\label{eq:busemann}
b_\xi(x)=\log\bigl(-\innerL{x}{\xi}\bigr);
\end{equation}
see, for example, \citet[Proposition~3.3]{BCCD23}. Indeed, let $y\in\HH^D$, let $v\in\T_y\HH^D$ be unit, and take $\xi=y+v$. Applying \eqref{eq:hyperboloid-distance} to the ray $\gamma(t)=\cosh t\,y+\sinh t\,v$, and using $\arccosh u=\log(2u)+o(1)$ as $u\to\infty$, gives
\[
\dist(x,\gamma(t))-t\longrightarrow\log\bigl(-\innerL{x}{y+v}\bigr).
\]
Conversely, every future-pointing null vector $\xi$ has the form $\xi=c(y+v)$ for some $c>0$, $y\in\HH^D$, and unit vector $v\in\T_y\HH^D$.

A geodesic $S\subseteq\HH^D$ and a point $y\in S$ determine two opposite geodesic rays emanating from $y$, and hence two ``opposite'' Busemann functions. The next lemma relates them to metric projection onto their shared geodesic $S$.
\begin{lemma}[Busemann coordinates around a geodesic]\label{lem:busemann-coordinates}
Let $y\in\HH^D$ and $v\in\T_y\HH^D$ be unit, and set
\[
S=\{\ell(s):s\in\RR\},\qquad
\ell(s)=\exp_y(sv)=\cosh s\,y+\sinh s\,v.
\]
For $x\in\HH^D$, write $\Proj_S(x)=\ell(s(x))$ and $\delta(x)=\dist(x,S)$. Then
\begin{equation}\tag{opposite Busemanns}\label{eq:opposite-busemanns}
\begin{aligned}
b_{y-v}(x)&=s(x)+\log\cosh\delta(x),\\
b_{y+v}(x)&=-s(x)+\log\cosh\delta(x).
\end{aligned}
\end{equation}
Here $s(x)=\inner{v}{\log_y(\Proj_S(x))}$.
\end{lemma}
\begin{proof}
Put $p=\ell(s(x))$ and $V=\spanof(y,v)$. By \eqref{eq:metric-projection}, $\Pi_Vx=\cosh\delta(x)\,p$. Since $y\pm v\in V$,
\[
-\innerL{x}{y\pm v}
=-\innerL{\Pi_Vx}{y\pm v}
=\cosh\delta(x)\bigl(-\innerL{p}{y\pm v}\bigr)
=\cosh\delta(x)e^{\mp s(x)}.
\]
The last equality uses $\innerL{p}{y\pm v}=-\cosh s(x)\pm\sinh s(x)=-e^{\mp s(x)}$. Taking logarithms proves both identities.
\end{proof}
Thus half the difference of the opposite Busemann functions gives position along $S$, while their average gives $\log\cosh\dist(\cdot,S)$. We use these identities to show that the lifting in \Cref{lem:lifting} is g-convex.

\begin{remark}[Why not use maxima of Busemann functions?]\label{rem:horospherical}
Nonnegatively scaled Busemann functions are g-convex and generalize affine functions, so why not build hard functions as their maxima? Finite-valued suprema of shifted, nonnegatively scaled Busemann functions form a \emph{strict} subclass of g-convex functions: the \emph{Busemann subdifferentiable}, or \emph{horospherically convex}, functions. For Lipschitz objectives in this subclass, algorithms have \emph{curvature-independent} $O(\eps^{-2})$ guarantees \citep{LLN24,GLLN26,CK25}. Thus maxima of Busemann functions \emph{cannot} yield curvature-dependent lower bounds. Here we use Busemann functions to prove \Cref{lem:lifting}, not to construct the hard function in \Cref{sec:construction}.
\end{remark}

\section{The nested geometry of Criscitiello and Boumal}\label{sec:geometry}
Fix $d\ge2$, $r>0$, and $\theta\in(0,\pi/2)$. We recall the inductive geometric construction in \citet[Appendix~A.2]{CB23}. Its objects are defined here both intrinsically and in the hyperboloid model; their stated geometric properties are those proved in that appendix. \Cref{tab:geometry-notation} summarizes the notation.

\begin{table}[!htbp]
\centering
\small
\renewcommand{\arraystretch}{1.12}
\begin{tabularx}{0.96\linewidth}{@{}lX@{}}
\toprule
\textbf{Symbol} & \textbf{Role} \\
\midrule
$I_k$ & Information submanifold. \\
$H_k$ & Totally geodesic submanifold normal to $I_k$ at $y_k$. \\
$y_k$ & Common metric projection of the points in $S_k$ onto $I_k$. \\
$e_i^{(k)}$ & Orthonormal frame at $y_k$, obtained by parallel transport. \\
$S_k$ & Sphere of remaining candidate minimizers. \\
$B_k$ & Closed ball in $H_k$ with boundary $S_k$. \\
$r_k$ & Uncertainty radius. \\
$\Delta_k$ & Step length $\dist(y_k,y_{k+1})$. \\
$x^*$ & Selected minimizer, belonging to $S_{d-1}$. \\
$u_k$ & Unit direction from $y_k$ toward the selected minimizer $x^*$. \\
$W_{j,\sigma}$ & Inward normal of a reflected correction halfspace (\Cref{subsec:ansatz}). \\
$L_{j,\sigma}$ & Reflected correction halfspace with inward normal $W_{j,\sigma}$ (\Cref{subsec:ansatz}). \\
\bottomrule
\end{tabularx}
\caption{Objects in the nested geometry and the subsequent ramp construction.}\label{tab:geometry-notation}
\end{table}

\paragraph{Base case.} Start with $y_0=\xref$, $I_0=\{y_0\}$, $H_0=\HH^d$, $r_0=r$, $B_0=B(y_0,r)$, and $S_0=\partial B_0$. Choose an orthonormal basis $e_1^{(0)},\ldots,e_d^{(0)}$ of $\T_{y_0}\HH^d$. Without loss of generality, choose hyperboloid coordinates so that
\[
y_0=\xref=e_0,\qquad e_i^{(0)}=e_i\quad(1\le i\le d),
\]
where $e_0,\ldots,e_d$ denote the standard basis of $\RR^{d+1}$.

\paragraph{Inductive step.} Suppose the objects have been constructed through stage $k$, where $0\le k\le d-2$. Inspired by \eqref{eq:right-triangle}, define $\Delta_k,r_{k+1}>0$ by
\begin{equation}\tag{def $r_k,\Delta_k$}\label{eq:recurrence}
\tanh\Delta_k=\cos\theta\tanh r_k,
\qquad
\sinh r_{k+1}=\sin\theta\sinh r_k.
\end{equation}
These satisfy $\cosh r_k=\cosh r_{k+1}\cosh\Delta_k$. Set
\begin{equation}\tag{def $y_k,e_i^{(k)}$}\label{eq:centers-frame}
y_{k+1}=\exp_{y_k}(\Delta_ke_{k+1}^{(k)}),
\qquad
e_i^{(k+1)}=P_{y_k\to y_{k+1}}e_i^{(k)}\quad(1\le i\le d).
\end{equation}
At each stage we define the information submanifold and its normal counterpart by
\begin{equation}\tag{def $I_k,H_k$}\label{eq:information-spaces}
I_k=\gspan(y_0,\ldots,y_k),\qquad
H_k=\gspan(y_k,e_{k+1}^{(k)},\ldots,e_d^{(k)}).
\end{equation}
Also define the set of remaining candidate minimizers by
\begin{equation}\tag{def $S_k$}\label{eq:Sk}
S_{k+1}=\{x\in S_k:
\inner{e_{k+1}^{(k)}}{\log_{y_k}(x)}=r_k\cos\theta\}.
\end{equation}
Thus the direction toward a remaining candidate makes angle $\theta$ with the next information direction.

In the hyperboloid model, these definitions give the following coordinate formulas \citep[Appendix~A.2, equation~(8) and Lemma~25]{CB23}:
\begin{align}
y_{k+1}&=\cosh\Delta_k\,y_k+\sinh\Delta_k\,e_{k+1},\notag\\
e_{k+1}^{(k+1)}&=\sinh\Delta_k\,y_k+\cosh\Delta_k\,e_{k+1},\notag\\
e_i^{(k)}&=e_i\quad(i>k),\tag{unexposed frame}\label{eq:unexposed-frame}\\
I_k&=\HH^d\cap\spanof(e_0,\ldots,e_k),\tag{information submanifolds}\label{eq:spaces-coord}\\
H_k&=\HH^d\cap\spanof(y_k,e_{k+1},\ldots,e_d)
=\HH^d\cap\bigcap_{i=1}^k(e_i^{(i)})^\perp.\notag
\end{align}
All coordinate orthogonal complements use the Lorentz form. To justify \eqref{eq:unexposed-frame}, consider one step from $y_q$ to $y_{q+1}$. The geodesic connecting $y_q$ and $y_{q+1}$ lies in $\spanof(y_q,e_{q+1})$, by \eqref{eq:exp} and \eqref{eq:centers-frame}. For $i>q+1$, the constant ambient vector $e_i$ is tangent to $\HH^d$ along that geodesic and has zero ambient derivative, hence zero covariant derivative. It is therefore unchanged by parallel transport; see also \citet[Appendix~A.1, equation~(5)]{CB23}. Induction from $e_i^{(0)}=e_i$ gives \eqref{eq:unexposed-frame}.

\begin{lemma}[Nested geometry, {\citealp[Lemmas~24--26]{CB23}}]\label{lem:geometry}
The totally geodesic submanifolds $I_k$ and $H_k$ have dimensions $k$ and $d-k$, meet at $y_k$, and have orthogonal tangent spaces there. The set $S_k$ is the (hyperbolic) sphere of radius $r_k$ centered at $y_k$ in $H_k$. If $B_k$ is the corresponding closed ball in $H_k$, then
\[
I_0\subseteq\cdots\subseteq I_{d-1},\qquad
H_0\supseteq\cdots\supseteq H_{d-1},\qquad
B_0\supseteq\cdots\supseteq B_{d-1},\qquad
S_0\supseteq\cdots\supseteq S_{d-1}.
\]

For every $x^*\in S_k$, the metric projection of $x^*$ onto $I_k$ is $y_k$. Moreover, for $0\le k\le d-2$ and every $x^*\in S_{k+1}$, the triangle
\[
(y_k,y_{k+1},x^*)
\]
has angle $\theta$ at $y_k$ and right angle at $y_{k+1}$. Therefore, using \eqref{eq:right-triangle},
\begin{equation}\tag{uncertainty radii}\label{eq:radii}
\sinh r_k=(\sin\theta)^k\sinh r,
\qquad
r_k\ge r/2\quad\text{if}\quad\cos^2\theta\le\tfrac12\text{ and }k\le\frac{\zeta_r}{2\cos^2\theta}.
\end{equation}
\end{lemma}

The radius $r_k$ is the radius of the remaining candidate sphere $S_k\subseteq H_k$ and, by \Cref{lem:geometry}, the distance from every remaining candidate minimizer to $I_k$. We therefore call the $r_k$ the \emph{uncertainty radii}: their slow decrease will prevent queries confined to $I_k$ from approaching the minimizer.

The estimate \eqref{eq:radii} is the iterative version of the slow radius decay highlighted in \Cref{subsec:triangles}. In Euclidean space, the same construction has $r_k=r(\sin\theta)^k$ and $\Delta_k=r_k\cos\theta$, so the corresponding horizon is only of order $1/\cos^2\theta$. \Cref{fig:overview} illustrates the nested geometry for $d=3$ in this Euclidean analogue.

\begin{figure}[!htbp]
\centering
\includegraphics[width=.88\linewidth]{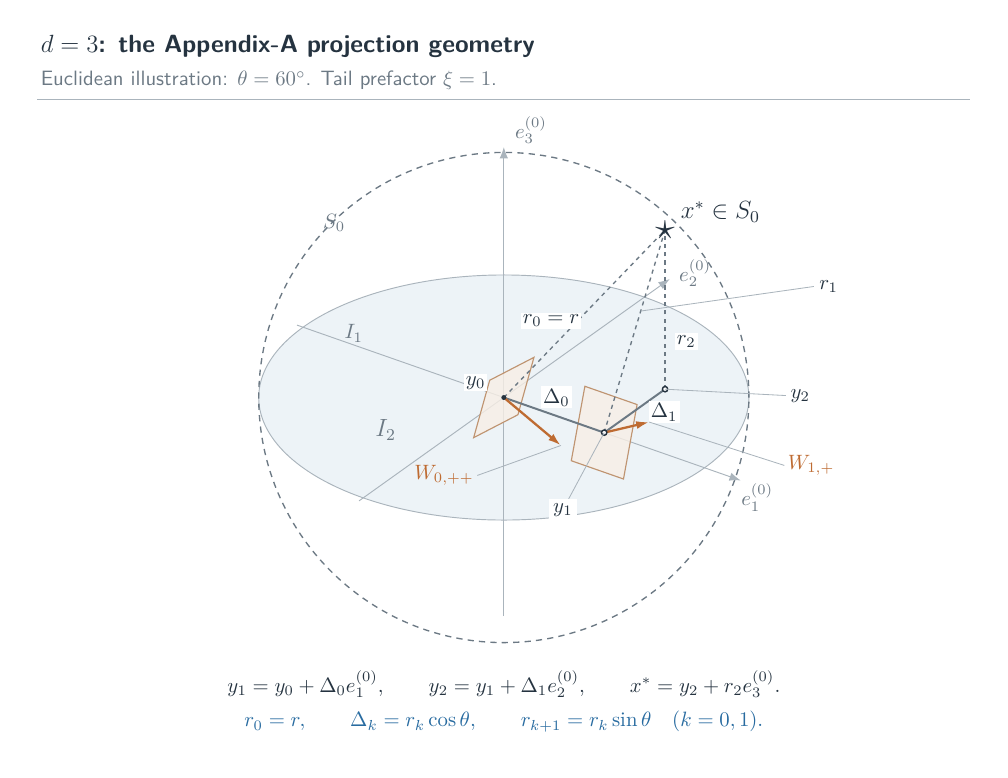}
\caption[Euclidean analogue of the nested geometry.]{Euclidean analogue of the nested geometry of \Cref{sec:geometry}, in dimension $d=3$ with $\theta=60^\circ$. The points $y_k$ are metric projections of $x^*\in S_0$ onto $I_k$. The information subspaces are $I_0=\{0\}$, $I_1=\spanof(e_1^{(0)})$, and $I_2=\spanof(e_1^{(0)},e_2^{(0)})$. The angle $\angle x^*y_ky_{k+1}$ equals $\theta$ for $k=0,1$.
\par\smallskip
Anticipating \Cref{sec:construction}, orange arrows show the normals $W_{0,++}$ and $W_{1,+}$ of the two small hyperplane patches in $\partial L_{0,++}$ and $\partial L_{1,+}$. The normals use \eqref{eq:W-reflected} and \eqref{eq:chosen-a}, except that the \emph{tail prefactor} $\xi$---the overall multiplicative factor in the coefficients $a_m$---is set to $\xi=1$ for visualization, rather than the proof value $\xi=1/4$.}\label{fig:overview}
\end{figure}

\paragraph{Choosing the minimizer $x^*$.} The set $S_{d-1}$ consists of two points. Choose $x^*\in S_{d-1}$ so that
\[
\frac{\log_{y_{d-1}}(x^*)}{r_{d-1}}=e_d^{(d-1)}.
\]
In the hyperboloid coordinates fixed above, $e_d^{(d-1)}=e_d$, so equivalently
\[
x^*=\cosh r_{d-1}\,y_{d-1}+\sinh r_{d-1}\,e_d.
\]
This point will be the minimizer of the hard function we construct. Write
\begin{equation}\tag{def $u_j$}\label{eq:uj}
u_j=\frac{\log_{y_j}(x^*)}{r_j}.
\end{equation}
By \Cref{lem:geometry}, the triangle $(y_j,y_{j+1},x^*)$ is right at $y_{j+1}$ and has angle $\theta$ at $y_j$. Parallel transport preserves inner products, so $P_{y_{j+1}\to y_j}u_{j+1}$ is a unit vector orthogonal to $e_{j+1}^{(j)}$. These vectors lie in the tangent plane of the triangle at $y_j$, while $u_j$ is a unit vector in the same plane making angle $\theta$ with $e_{j+1}^{(j)}$. Hence
\[
u_j=\cos\theta\,e_{j+1}^{(j)}+\sin\theta\,P_{y_{j+1}\to y_j}u_{j+1}.
\]
Iterating this identity, using \eqref{eq:centers-frame}, yields
\begin{equation}\tag{$u_j$ expansion}\label{eq:direction-expansion}
u_j=\cos\theta\sum_{\ell=j+1}^{d-1}(\sin\theta)^{\ell-j-1}e_\ell^{(j)}
+(\sin\theta)^{d-j-1}e_d^{(j)}.
\end{equation}
In the hyperboloid model, $e_\ell^{(j)}=e_\ell$ for $\ell\ge j+1$ by \eqref{eq:unexposed-frame}, so \eqref{eq:direction-expansion} is also the ambient coordinate expansion of $u_j$. The last coefficient has no factor $\cos\theta$ because the recursion terminates with a unit vector. The squared coefficients sum to one, as they must.

\begin{remark}[Connection with Polyak subgradient descent]\label{rem:polyak}
The preceding geometry was originally motivated by the geometric interpretation of the Polyak subgradient method \citep[Appendix~A.3]{CB23}. Suppose at iteration $k$ that a minimizer is known to belong to $B(x_k,s_k)$. After querying $x_k$ and receiving $g_k\in\partial f(x_k)$, the subgradient inequality \eqref{eq:subgradient} further localizes the minimizer to
\[
B(x_k,s_k)\cap
\{x:\inner{g_k}{\log_{x_k}(x)}\le f^*-f(x_k)\}.
\]
The hyperbolic Polyak update chooses $x_{k+1}$ precisely as the center of the smallest ball containing this intersection, with $s_{k+1}$ its radius.

Starting from $x_0=\xref$ with $s_0=r$, run this Polyak method on the hard function of \citet[Appendix~A]{CB23}, whose correction halfspaces are recalled in \Cref{subsec:ansatz,rem:old}. Then the queries $x_k$ of the method exactly correspond to the points $y_k$, while the $s_k$ are precisely the successive uncertainty radii $r_k$:
\[
x_k=y_k,\qquad s_k=r_k.
\]
The recurrence in \Cref{lem:geometry} is therefore both the geometry underlying our lower bound and the worst-case ``ball-shrinking'' recurrence of the Polyak subgradient method.
\end{remark}

\section{Reflected ramps and the lower bound}\label{sec:construction}
We retain the geometry of \Cref{sec:geometry}, with $d\ge2$, $r>0$, and $\theta\in(0,\pi/2)$. We use the minimizer $x^*$, the information submanifolds $I_j$, and the orthonormal frames $e_i^{(j)}$ from \Cref{sec:geometry}. We also continue to use the hyperboloid coordinates fixed in \Cref{sec:geometry}, so that
\[
\xref=y_0=e_0,\qquad e_i^{(0)}=e_i,
\]
and in particular the coordinate identities \eqref{eq:unexposed-frame} and \eqref{eq:spaces-coord} hold throughout this section.

In \Cref{subsec:ansatz} we introduce the hard-function ansatz and its reflected halfspaces, leaving their coefficients undetermined. \Cref{subsec:upperrequirements} imposes upper bounds on these coefficients so that the halfspaces contain the minimizer strictly in their interiors, with a quantitative margin. In \Cref{subsec:completion}, a complementary lower requirement ensures that the reflected ramps can cancel the components of a subgradient normal to an information submanifold. We then choose coefficients satisfying both upper and lower requirements, and deduce the advance-or-stall lemma. Finally, \Cref{subsec:A1proof} proves the lower bound under a mild geodesic-span assumption. The black-box reduction in \Cref{sec:unrestricted} removes that assumption and completes the proof of \Cref{thm:main}.

\subsection{The hard-function ansatz}\label{subsec:ansatz}
Let $a_1,a_2,\ldots>0$ be coefficients to be determined. For $j=0,\ldots,d-2$, define
\begin{equation}\tag{base normal}\label{eq:Wbase}
\begin{aligned}
W_j={}&e_{j+1}^{(j)}-\sum_{\ell=j+2}^{d-1}a_{\ell-j-1}e_\ell^{(j)}\\
&-a_{d-j-1}\sec\theta\,e_d^{(j)}
\quad\in\T_{y_j}\HH^d.
\end{aligned}
\end{equation}
The terminal factor $\sec\theta$ compensates for the missing factor $\cos\theta$ in the last component of $u_j$ in \eqref{eq:direction-expansion}. By \eqref{eq:unexposed-frame}, all frame vectors $e_\ell^{(j)}$ appearing here equal their ambient counterparts $e_\ell$, so \eqref{eq:Wbase} is also the coordinate expansion of $W_j$ in the hyperboloid model.

For a sign tuple
\[
\sigma=(\sigma_{j+2},\ldots,\sigma_d)\in\{-1,+1\}^{d-j-1},
\]
define the ``reflected normal''
\begin{equation}\tag{reflected normals}\label{eq:W-reflected}
\begin{aligned}
W_{j,\sigma}={}&\underbrace{e_{j+1}^{(j)}}_{\text{head}}
-\underbrace{\sum_{\ell=j+2}^{d-1}a_{\ell-j-1}\sigma_\ell e_\ell^{(j)}}_{\text{nonterminal tail}}\\
&-\underbrace{a_{d-j-1}\sec\theta\,\sigma_d e_d^{(j)}}_{\text{terminal tail}}
\quad\in\T_{y_j}\HH^d.
\end{aligned}
\end{equation}
We call the $e_{j+1}^{(j)}$ term the \emph{head}. The remaining, subtracted terms form the \emph{tail}: the sum over $\ell=j+2,\ldots,d-1$ is the \emph{nonterminal tail}, and the term in direction $e_d^{(j)}$ is the \emph{terminal tail}. We then define the corresponding hyperbolic halfspace using \eqref{eq:halfspace}:
\begin{equation}\tag{correction halfspaces}\label{eq:L-reflected}
\begin{aligned}
L_{j,\sigma}
&=\exp_{y_j}\{v\in\T_{y_j}\HH^d:\inner{W_{j,\sigma}}{v}\ge0\}\\
&=\{x\in\HH^d:\innerL{W_{j,\sigma}}{x}\ge0\}.
\end{aligned}
\end{equation}
By \eqref{eq:information-spaces} and \Cref{lem:geometry}, the vectors $e_{j+1}^{(j)},\ldots,e_d^{(j)}$ span $\T_{y_j}H_j$, which is orthogonal to $\T_{y_j}I_j$. Thus $W_{j,\sigma}\in\T_{y_j}H_j$ is orthogonal to $\T_{y_j}I_j$. Since $I_j=\exp_{y_j}(\T_{y_j}I_j)$, the definition \eqref{eq:L-reflected} gives
\begin{equation}\label{eq:boundary-anchor}
I_j\subseteq\partial L_{j,\sigma}.
\end{equation}
Orthonormality shows that $\norm{W_{j,\sigma}}=\norm{W_j}$, independently of $\sigma$.

For a coefficient $\alpha>0$ to be selected, let
\begin{equation}\tag{ramp construction}\label{eq:core}
\begin{aligned}
P(x)&=\max_{0\le j\le d-2}\ \max_{\sigma\in\{-1,+1\}^{d-j-1}}\dist(x,L_{j,\sigma}),\\
f(x)&=\dist(x,x^*)+\alpha P(x).
\end{aligned}
\end{equation}
The inner maximum is over a different sign tuple for each $j$.\footnote{More formally, one may write $\sigma^{(j)}\in\{-1,+1\}^{d-j-1}$ in the stage-$j$ maximum. We suppress the superscript and simply write $\sigma$.} This ansatz for $f$ is in part motivated by the distance-to-point plus distance-to-halfspace construction in \Cref{prop:sharp-support}. By \Cref{subsec:gconvex}, the function $f$ is g-convex and $(1+\alpha)$-Lipschitz. We will ensure that every halfspace contains $x^*$ (\Cref{lem:margin}); then $f^*=f(x^*)=0$, uniquely, because $f(x)\ge\dist(x,x^*)$.

\begin{remark}[The old normals and their coefficients]\label{rem:old}
\citet[Appendix~A]{CB23} use the specific normal
\[
W_j^{\mathrm{old}}
=\frac{e_{j+1}^{(j)}-\cos\theta\,u_j}{\sin^2\theta}.
\]
By \eqref{eq:direction-expansion}, $\inner{W_j^{\mathrm{old}}}{u_j}=0$, and $W_j^{\mathrm{old}}\in\spanof(e_{j+1}^{(j)},u_j)\subseteq\T_{y_j}H_j$. Consequently, the corresponding hyperbolic hyperplane
\[
\partial L_j^{\mathrm{old}}
=\exp_{y_j}\{v\in\T_{y_j}\HH^d:\inner{W_j^{\mathrm{old}}}{v}=0\}
\]
contains both $I_j$ and $x^*=\exp_{y_j}(r_ju_j)$. The corresponding coefficients are
\begin{equation}\tag{old coefficients}\label{eq:old-a}
a_m^{\mathrm{old}}
=\frac{\cos^2\theta(\sin\theta)^m}{\sin^2\theta},
\end{equation}
with terminal coefficient $a_m^{\mathrm{old}}\sec\theta$.

This choice of halfspaces gives the $\Omega(\zeta_r/\eps^2)$ lower bound for the Polyak subgradient method in \citet[Appendix~A]{CB23}. Augmenting with reflections, as in \eqref{eq:W-reflected} and \eqref{eq:core}, gives an $\Omega(\eps^{-2})$ guarantee through the advance-or-stall argument below, but does not recover the curvature factor; see \citet[Chapter~5, Remark~5.1.17]{CriscitielloThesis25}. To recover that factor, we must also change the normals. \Cref{fig:d2} compares the original and modified halfspaces in the two-dimensional Euclidean analogue.
\end{remark}

\begin{figure}[!htbp]
\centering
\includegraphics[width=\linewidth]{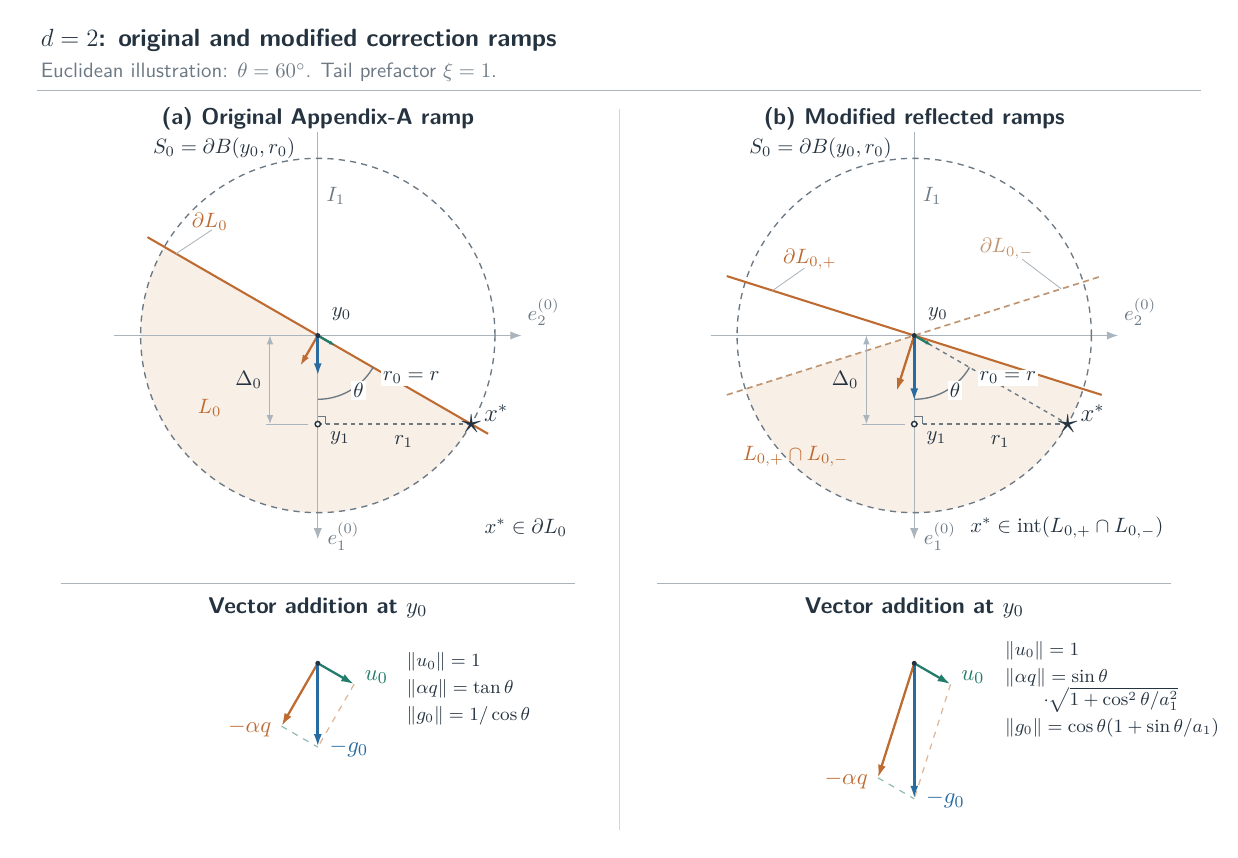}
\caption[Original and modified correction ramps.]{Comparison of the original ramps proposed by Criscitiello and Boumal (left) and our modified reflected ramps (right), in the Euclidean analogue of the construction. The parameters are $d=2$, $\theta=60^\circ$, and tail prefactor $\xi=1$, rather than the proof value $\xi=1/4$. The original halfspace $L_0$ has $x^*$ on its boundary; both modified halfspaces contain $x^*$ in their interiors, illustrating \Cref{lem:margin}.
\par\smallskip
The lower panels show the oracle at $y_0$ for the corresponding Euclidean hard function of the form \eqref{eq:core}. The negative gradient of $\dist(\cdot,x^*)$ is the unit direction $u_0$ toward $x^*$ (green). The ramp correction contributes $-\alpha q$ (orange), where $q$ is a chosen subgradient of the ramp maximum, as in the proof of \Cref{lem:completion}; here it is normal to an active ramp boundary. Their sum $-g_0=u_0-\alpha q$ (blue) is the negative of the returned subgradient and lies in $I_1=\spanof(e_1^{(0)})$.}\label{fig:d2}
\end{figure}

\subsubsection*{The reflection isometries}
The sign tuples in \eqref{eq:W-reflected} have a geometric meaning: they describe halfspaces obtained by reflecting the base halfspace. It is useful to make this explicit.

For each $1\le\ell\le d$, consider the ambient linear map $Q_\ell:\RR^{d+1}\to\RR^{d+1}$ that reverses the $\ell$th spatial coordinate and leaves all other coordinates unchanged:
\[
e_\ell\longmapsto-e_\ell,\qquad e_i\longmapsto e_i\quad(i\ne\ell).
\]
Equivalently, its formula is
\begin{equation}\tag{reflection}\label{eq:reflection}
Q_\ell x=x-2\innerL{x}{e_\ell}e_\ell.
\end{equation}
This map preserves the Lorentz inner product and the positive sheet in \eqref{eq:model}, so it is an isometry of $\HH^d$. More precisely, it is reflection across the codimension-one totally geodesic hyperplane
\[
K_\ell=\HH^d\cap e_\ell^\perp.
\]
By \eqref{eq:spaces-coord}, $I_{\ell-1}\subseteq K_\ell$, so $Q_\ell$ fixes $I_{\ell-1}$ pointwise. In particular, when $\ell\ge j+2$, it fixes $y_j$ and the head $e_{j+1}^{(j)}$, and reverses precisely the $\ell$th tail coordinate of \eqref{eq:W-reflected}.

Let $L_j=L_{j,(+1,\ldots,+1)}$ be the base halfspace associated with $W_j$ in \eqref{eq:Wbase}. Because $Q_\ell$ is an isometry fixing $y_j$, it sends the halfspace with normal $W_{j,\sigma}$ to the halfspace with normal $Q_\ell W_{j,\sigma}$. Thus it changes precisely the sign $\sigma_\ell$. The full family
\[
\{L_{j,\sigma}:\sigma\in\{-1,+1\}^{d-j-1}\}
\]
is the orbit of $L_j$ under the commuting reflections $Q_{j+2},\ldots,Q_d$.

One can also generate these families successively. Start with $L_0$ and reflect it across $K_2$, which contains $I_1$, to obtain both choices of $\sigma_2$. Next introduce $L_1$ and reflect the existing halfspaces across $K_3$, which contains $I_2$. This introduces the $\sigma_3$ sign: the stage-zero family now has four copies, and the stage-one family has two. Continuing in this way adds one independent sign to each earlier family at every step.

\paragraph{Which reflected copies are simultaneously active?}
Fix a point $y\in I_k$ with $k\le d-1$ and an earlier stage $j<k$. Equation \eqref{eq:spaces-coord} gives
\[
\innerL{y}{e_\ell}=0\qquad(\ell>k).
\]
Using \eqref{eq:W-reflected}, we therefore have
\[
\innerL{W_{j,\sigma}}{y}
=\innerL{e_{j+1}}{y}
-\sum_{\ell=j+2}^{k}a_{\ell-j-1}\sigma_\ell\innerL{e_\ell}{y}.
\]
The terminal term is absent because $k\le d-1$. Moreover, $\norm{W_{j,\sigma}}=\norm{W_j}$ is independent of all signs. The distance formula \eqref{eq:signed-distance} consequently shows that
\[
\dist(y,L_{j,\sigma})
=\arcsinh\!\left(\frac{\positive{-\innerL{W_{j,\sigma}}{y}}}{\norm{W_j}}\right)
\]
depends on the \emph{exposed signs} $\sigma_{j+2},\ldots,\sigma_k$, but not on the \emph{future signs} $\sigma_{k+1},\ldots,\sigma_d$.

Suppose one pair $(j,\sigma)$ is active, meaning that $\dist(y,L_{j,\sigma})=P(y)$. Then every sign tuple $\sigma'$ satisfying
\begin{equation}\tag{matching prefix}\label{eq:active-prefix}
\sigma'_\ell=\sigma_\ell\quad(j+2\le\ell\le k),
\qquad
\sigma'_\ell\in\{-1,+1\}\quad(\ell>k)
\end{equation}
is also active. There are $2^{d-k}$ such extensions of the exposed prefix. Their free signs correspond to the directions $e_{k+1},\ldots,e_d$; these form a normal basis of $I_k$ at $y$, as we verify in \Cref{lem:completion}. Thus an active old ramp leaves every future sign free, although different choices of the exposed prefix need not be active. We use this freedom to choose a convex combination of ramp subgradients in \Cref{lem:completion}.

\subsection{Upper requirements on \texorpdfstring{$a_m$}{the coefficients}: containment and interior margin}\label{subsec:upperrequirements}
To obtain $f(x^*)=0$ in \eqref{eq:core}, every halfspace $L_{j,\sigma}$ must contain $x^*$. We will place $x^*$ strictly in their interiors: the resulting separation will be crucial when an old ramp is active. We also need to control $\norm{W_{j,\sigma}}$. The following two pointwise upper bounds suffice:
\begin{equation}\tag{upper requirements}\label{eq:upper-a}
a_m\le\frac14\cos^2\theta,
\qquad
a_m\le\frac14\cos\theta(\sin\theta)^m
\qquad(m\ge1).
\end{equation}
The numerical factor $1/4$ is chosen for convenience. Any fixed factor strictly between $0$ and $1/3$ works.

The two bounds in \eqref{eq:upper-a} play complementary roles. The second controls the size of the tail of $W_{j,\sigma}$, which keeps each reflected normal close to its unit head $e_{j+1}^{(j)}$ and hence keeps $\norm{W_{j,\sigma}}$ bounded. The first controls the interaction of the tail of $W_{j,\sigma}$ with the direction $u_j$: the head contributes $\cos\theta$ to $\inner{W_{j,\sigma}}{u_j}$, and this bound prevents the tail from cancelling more than a constant fraction of that contribution. Consequently, $W_{j,\sigma}$ remains positively aligned with $u_j$, which places $x^*$ strictly inside $L_{j,\sigma}$.

\begin{lemma}[Geometry of a correction halfspace]\label{lem:margin}
Under \eqref{eq:upper-a}, for every $j,\sigma$,
\begin{equation}\label{eq:norm-margin}
1\le\norm{W_{j,\sigma}}<17/16,
\qquad
\inner{W_{j,\sigma}}{u_j}\ge\tfrac14\cos\theta.
\end{equation}
Additionally, $y_k\in L_{j,\sigma}$ for all $j\le k\le d-1$, and $x^*$ belongs to the interior of $L_{j,\sigma}$, with
\begin{equation}\tag{interior margin}\label{eq:intrinsic-margin}
\sinh\dist(x^*,\partial L_{j,\sigma})
\ge\frac15\cos\theta\,\sinh r_j.
\end{equation}
\end{lemma}
\Needspace{5\baselineskip}
\begin{proof}
We verify each claim in turn.

\paragraph{Bounding the normal length.}
The frame vectors $e_{j+1}^{(j)},\ldots,e_d^{(j)}$ from \eqref{eq:centers-frame} are orthonormal in $\T_{y_j}\HH^d$. Hence \eqref{eq:W-reflected} gives
\begin{align*}
1\le\norm{W_{j,\sigma}}^2
&=1+\sum_{m=1}^{d-j-2}a_m^2+a_{d-j-1}^2\sec^2\theta\\
&\le1+\frac{\cos^2\theta}{16}\sum_{m\ge1}(\sin\theta)^{2m}+\frac1{16}\\
&=1+\frac{\sin^2\theta}{16}+\frac1{16}
\le\frac98<\left(\frac{17}{16}\right)^2.
\end{align*}
This uses the second inequality in \eqref{eq:upper-a}.

\paragraph{Positive alignment and interior margin.}
We now compute $\inner{W_{j,\sigma}}{u_j}$ using the expansion \eqref{eq:direction-expansion}. The head $e_{j+1}^{(j)}$ in \eqref{eq:W-reflected} contributes $\cos\theta$. The contribution of the nonterminal tail has absolute value at most
\begin{align*}
\left|\cos\theta\sum_{m=1}^{d-j-2}a_m\sigma_{j+m+1}(\sin\theta)^m\right|
&\le\cos\theta\sum_{m\ge1}a_m(\sin\theta)^m\\
&\le\frac{\cos^3\theta}{4}\frac{\sin\theta}{1-\sin\theta}\\
&=\frac{\cos\theta}{4}\sin\theta(1+\sin\theta)
\le\frac12\cos\theta,
\end{align*}
where the second inequality uses the first upper bound in \eqref{eq:upper-a}. The terminal tail contributes a term of absolute value
\[
a_{d-j-1}\sec\theta(\sin\theta)^{d-j-1}\le\tfrac14\cos\theta,
\]
again by the first inequality in \eqref{eq:upper-a}. Combining the head, nonterminal tail, and terminal tail gives
\[
\inner{W_{j,\sigma}}{u_j}
\ge\cos\theta-\tfrac12\cos\theta-\tfrac14\cos\theta
=\tfrac14\cos\theta,
\]
proving the second inequality in \eqref{eq:norm-margin}.

Let $\nu_{j,\sigma}=W_{j,\sigma}/\norm{W_{j,\sigma}}$. We therefore have
\[
\inner{\nu_{j,\sigma}}{u_j}
\ge\frac4{17}\cos\theta\ge\frac15\cos\theta>0.
\]
By \eqref{eq:uj}, $x^*=\exp_{y_j}(r_ju_j)$. The strict inequality above therefore places $x^*$ in $\interior L_{j,\sigma}$ by \eqref{eq:L-reflected}. Let $m=\dist(x^*,\partial L_{j,\sigma})$. The metric projection of $x^*$ onto $\partial L_{j,\sigma}$ gives a right triangle with hypotenuse $r_j$. The sine rule \eqref{eq:right-triangle} gives
\[
\sinh m=\sinh r_j\inner{\nu_{j,\sigma}}{u_j},
\]
which proves \eqref{eq:intrinsic-margin}.

\paragraph{Containment of the centers.}
It remains to verify $y_k\in L_{j,\sigma}$ for $j\le k\le d-1$. When $k=j$, this follows from $y_j\in I_j\subseteq\partial L_{j,\sigma}$, as in \eqref{eq:boundary-anchor}. For $k>j$, recall from \eqref{eq:exp} and \eqref{eq:uj} that
\[
x^*=\cosh r_j\,y_j+\sinh r_j\,u_j
=\cosh r_k\,y_k+\sinh r_k\,u_k.
\]
For $j<\ell\le k\le d-1$, take the Lorentz inner product of these expressions with $e_\ell$. Since $\innerL{y_j}{e_\ell}=0$ for $\ell>j$, while $\innerL{u_k}{e_\ell}=0$ for $\ell\le k$, we obtain
\[
\sinh r_j\,\innerL{u_j}{e_\ell}
=\cosh r_k\,\innerL{y_k}{e_\ell}.
\]
Using \eqref{eq:direction-expansion},
\[
\innerL{u_j}{e_\ell}=\cos\theta(\sin\theta)^{\ell-j-1},
\]
and therefore
\[
\innerL{y_k}{e_\ell}
=\frac{\cos\theta(\sin\theta)^{\ell-j-1}\sinh r_j}{\cosh r_k}.
\]
Combining this identity with \eqref{eq:W-reflected}, we obtain
\begin{equation}\label{eq:center-contained}
\innerL{W_{j,\sigma}}{y_k}
=\frac{\cos\theta\sinh r_j}{\cosh r_k}
\left(1-\sum_{m=1}^{k-j-1}a_m\sigma_{j+m+1}(\sin\theta)^m\right).
\end{equation}
There is no terminal term because $y_k$ has zero $e_d$ coordinate by \eqref{eq:spaces-coord}. Using the first inequality in \eqref{eq:upper-a}, the sum in parentheses has absolute value at most
\[
\sum_{m\ge1}a_m(\sin\theta)^m
\le\frac{\cos^2\theta}{4}\frac{\sin\theta}{1-\sin\theta}
=\tfrac14\sin\theta(1+\sin\theta)\le\tfrac12.
\]
Thus \eqref{eq:center-contained} is positive, and \eqref{eq:L-reflected} gives $y_k\in L_{j,\sigma}$.
\end{proof}

The interior margin \eqref{eq:intrinsic-margin} distinguishes these ramps from the original corrections of \citet{CB23}. An original correction halfspace has $x^*$ on its boundary, as explained in \Cref{rem:old}, and therefore gives no interior margin separating a query from $x^*$. The modified halfspaces provide the definite separation in \eqref{eq:intrinsic-margin}. This is crucial for the next lemma.

\begin{lemma}[Old-ramp geometry]\label{lem:oldramp}
Assume \eqref{eq:upper-a}. Let $0\le j\le k\le d-2$, let $y\in I_k$, and suppose $y\notin\interior L_{j,\sigma}$. With
\[
D=\dist(y,x^*),\qquad \delta=\dist(y,L_{j,\sigma}),
\]
we have
\begin{equation}\tag{old-ramp estimate}\label{eq:oldramp}
\frac{\sinh r_j\cosh\delta}{\sinh D}
\le\frac{6}{(\sin\theta)^{k-j}+\cos\theta}.
\end{equation}
The statement includes a boundary point, with $\delta=0$.
\end{lemma}
\Needspace{5\baselineskip}
\begin{proof}
This follows from two estimates, both grounded in geometry.

\paragraph{Right-triangle estimate.}
First, $y_k\in L_{j,\sigma}$ by \Cref{lem:margin}, so $\dist(y,y_k)\ge\delta$. Since $y_k$ is the metric projection of $x^*$ onto $I_k$ by \Cref{lem:geometry}, the right-triangle identity \eqref{eq:right-triangle} gives
\[
\cosh D=\cosh r_k\,\cosh\dist(y,y_k).
\]
Using $\cosh^2(\cdot)=1+\sinh^2(\cdot)$, we obtain
\[
\sinh^2D
=\sinh^2r_k\cosh^2\dist(y,y_k)+\sinh^2\dist(y,y_k).
\]
Consequently $\sinh D\ge\sinh r_k\cosh\delta$. Using \eqref{eq:radii}, we obtain
\begin{equation}\label{eq:oldramp-radius}
\frac{\sinh r_j\cosh\delta}{\sinh D}
\le\frac{\sinh r_j}{\sinh r_k}
=\frac1{(\sin\theta)^{k-j}}.
\end{equation}

\paragraph{Hyperplane-separation estimate.}
Second, let $m=\dist(x^*,\partial L_{j,\sigma})$. The signed distance to a hyperplane is $1$-Lipschitz. Its values at $x^*$ and $y$ are $m$ and $-\delta$, respectively, so $D\ge m+\delta$. Using the $\sinh$ addition formula and \eqref{eq:intrinsic-margin},
\[
\sinh D\ge\sinh(m+\delta)\ge\sinh m\cosh\delta
\ge\tfrac15\cos\theta\sinh r_j\cosh\delta.
\]
Thus the same ratio satisfies
\begin{equation}\label{eq:oldramp-margin}
\frac{\sinh r_j\cosh\delta}{\sinh D}\le\frac5{\cos\theta}.
\end{equation}

Combining \eqref{eq:oldramp-radius} and \eqref{eq:oldramp-margin}, and using $\min\{a^{-1},5b^{-1}\}\le6/(a+b)$ for $a,b>0$, proves \eqref{eq:oldramp}.\footnote{This final interpolation is optional: one may instead retain the minimum of the two geometric bounds throughout the subsequent argument, and choose the coefficients accordingly. We use the displayed harmonic interpolation only to avoid carrying repeated minimum expressions through later formulas.}
\end{proof}

The second term in the denominator of \eqref{eq:oldramp} comes from \eqref{eq:intrinsic-margin}. Even when the active ramp $j$ is very old ($j\ll k$), the estimate
\[
\sinh\dist(y,x^*)\ge\tfrac15\sinh r_j\cosh\delta\cos\theta
\]
keeps a potential query $y\in I_k$ well-separated from $x^*$ on a scale determined by the ramp's radius $r_j$, rather than only the current radius $r_k$. This separation is what allows the stall case of \Cref{lem:stall-advance} to work regardless of the active ramp's age $k-j$.

\subsection{Choice of coefficients and advance-or-stall}\label{subsec:completion}
The next lemma is the main step in constructing the oracle. Its lower requirement on the coefficients allows the reflected ramps to cancel the ``hidden'' part of the gradient of $\dist(\cdot,x^*)$: the component normal to the information submanifold in which we want the oracle answer to lie. We then choose coefficients satisfying both upper and lower requirements and deduce the advance-or-stall lemma.

\begin{lemma}[Cancelling the hidden components]\label{lem:completion}
Let $a_m>0$ and $\alpha>0$, and define $P$ and $f$ by \eqref{eq:core}. Assume \eqref{eq:upper-a} and also
\begin{equation}\tag{lower requirements}\label{eq:lower-a}
a_m\ge\frac{51}{8\alpha}
\frac{\cos\theta(\sin\theta)^m}{(\sin\theta)^m+\cos\theta}
\qquad(m\ge1).
\end{equation}
Let $y\in I_k$ with $0\le k\le d-2$. Suppose $(j,\sigma)$ is active in $P(y)$, meaning $\dist(y,L_{j,\sigma})=P(y)$, with $j\le k$, and that $y\notin\interior L_{j,\sigma}$. Then there is a subgradient $g\in\partial f(y)$ tangent to $I_{\max\{k,j+1\}}$.
\end{lemma}
\begin{proof}
Write $\sigma=(\sigma_{j+2},\ldots,\sigma_d)$. We use $\sigma'\in\{-1,+1\}^{d-j-1}$ for sign tuples agreeing with $\sigma$ for $j+2\le\ell\le k$, with the remaining signs free. When $j<k$, every such pair $(j,\sigma')$ is active at $y$ by \eqref{eq:active-prefix}, that is,
\[
\dist(y,L_{j,\sigma'})=P(y)=\dist(y,L_{j,\sigma}).
\] When $j=k$, every halfspace in the ``fresh'' family $(k,\sigma')$ has $y\in I_k$ on its boundary by \eqref{eq:boundary-anchor}. Since the chosen pair is active, $P(y)=0$, so all these fresh copies are active as well.

Put
\[
D=\dist(y,x^*),\qquad
\delta=\dist(y,L_{j,\sigma}),\qquad
u=\frac{\log_y(x^*)}{D}.
\]
Here $D\ge r_k>0$ by \Cref{lem:geometry}. We first identify the directions to be cancelled. By \eqref{eq:spaces-coord}, the tangent space of $I_{\max\{k,j+1\}}$ at $y$ is the intersection of $\T_y\HH^d$ with $\spanof(e_0,\ldots,e_{\max\{k,j+1\}})$. Consequently, the vectors
\begin{equation}\tag{hidden directions}\label{eq:free-directions}
e_\ell,\qquad \ell>\max\{k,j+1\},
\end{equation}
are tangent at $y$ and form an orthonormal basis of its normal space. At $y$, the gradient of the term $\dist(\cdot,x^*)$ in \eqref{eq:core} is $-u$. We shall cancel the components of $-u$ in these directions, so that the corrected subgradient is tangent to $I_{\max\{k,j+1\}}$.

\paragraph{Components to be cancelled.}
First compute how large these components are. Formula \eqref{eq:log} gives
\[
u=\frac{x^*-\cosh D\,y}{\sinh D}.
\]
For $\ell>\max\{k,j+1\}$, put $m=\ell-j-1$; then $m\ge1$ and $m\ge k-j$.

Since $y\in I_k$ and $y_j\in I_j\subseteq I_k$, both $y$ and $y_j$ are Lorentz-orthogonal to $e_\ell$, by \eqref{eq:spaces-coord}. Consequently, using $x^*=\cosh r_j\,y_j+\sinh r_j\,u_j$, we obtain
\[
\innerL{u}{e_\ell}
=\frac{\innerL{x^*}{e_\ell}}{\sinh D}
=\frac{\sinh r_j}{\sinh D}\innerL{u_j}{e_\ell}.
\]
Thus the hidden components of $u$ at $y$ are obtained from those of $u_j$ at $y_j$ by a common scaling factor. Using \eqref{eq:direction-expansion}, we conclude that
\begin{equation}\label{eq:base-hidden-components}
\begin{aligned}
\innerL{u}{e_\ell}
&=\frac{\cos\theta(\sin\theta)^m\sinh r_j}{\sinh D}
&& \bigl(\max\{k,j+1\}<\ell<d\bigr),\\
\innerL{u}{e_d}
&=\frac{(\sin\theta)^{d-j-1}\sinh r_j}{\sinh D}.
\end{aligned}
\end{equation}
\Cref{lem:oldramp} gives, for $\max\{k,j+1\}<\ell<d$,
\begin{equation}\label{eq:required-component}
\begin{aligned}
\left|\innerL{u}{e_\ell}\right|\cosh\delta
&\le\frac{6\cos\theta(\sin\theta)^m}{(\sin\theta)^{k-j}+\cos\theta}\\
&\le\frac{6\cos\theta(\sin\theta)^m}{(\sin\theta)^m+\cos\theta}
\le\frac{\alpha a_m}{\norm{W_j}}.
\end{aligned}
\end{equation}
The second inequality uses $m\ge k-j$. The last uses \eqref{eq:lower-a} and $\norm{W_j}<17/16$ from \Cref{lem:margin}, since $(51/8)(16/17)=6$. For $\ell=d$, the same computation gives
\begin{equation}\label{eq:required-terminal}
\left|\innerL{u}{e_d}\right|\cosh\delta
\le\frac{6(\sin\theta)^{d-j-1}}{(\sin\theta)^{d-j-1}+\cos\theta}
\le\frac{\alpha a_{d-j-1}\sec\theta}{\norm{W_j}}.
\end{equation}
The extra terminal factor $\sec\theta$ in \eqref{eq:W-reflected} thus compensates exactly for the larger terminal component in \eqref{eq:base-hidden-components}.

\paragraph{Components supplied by the ramps.}
Next identify the $e_\ell$ components supplied by the reflected ramps $\dist(\cdot,L_{j,\sigma'})$. For each copy $\sigma'$ above, let $v_{j,\sigma'}(y)$ be the gradient of this distance function at $y$ if $\delta>0$, and the outward unit normal to $L_{j,\sigma'}$ if $\delta=0$. In either case it is a valid subgradient at $y$, and \eqref{eq:ramp-gradient} gives
\[
v_{j,\sigma'}(y)
=-\frac{W_{j,\sigma'}/\norm{W_j}-\sinh\delta\,y}{\cosh\delta}.
\]
Because $y\in I_k$, we have $\innerL{y}{e_\ell}=0$ for $\ell>k$. Equation \eqref{eq:W-reflected} therefore gives
\begin{equation}\label{eq:available-component}
\begin{aligned}
\innerL{v_{j,\sigma'}(y)}{e_\ell}
&=\frac{a_m\sigma'_\ell}{\norm{W_j}\cosh\delta}
&& \bigl(\max\{k,j+1\}<\ell<d\bigr),\\
\innerL{v_{j,\sigma'}(y)}{e_d}
&=\frac{a_{d-j-1}\sec\theta\,\sigma'_d}{\norm{W_j}\cosh\delta}.
\end{aligned}
\end{equation}
Every independent choice of the signs displayed here occurs among the active copies.

\paragraph{Choosing the convex combination.}
Let $\Pi^\perp$ be orthogonal projection from $\T_y\HH^d$ onto the normal space of $I_{\max\{k,j+1\}}$ at $y$:
\[
\Pi^\perp w=\sum_{\ell>\max\{k,j+1\}}\innerL{w}{e_\ell}e_\ell.
\]
By \eqref{eq:available-component}, the vectors $\Pi^\perp(\alpha v_{j,\sigma'}(y))$ form the vertices of a rectangular box centered at zero. Equations \eqref{eq:required-component} and \eqref{eq:required-terminal} put $\Pi^\perp u$ inside that box. Since a box is the convex hull of its vertices, there are weights $\omega_{\sigma'}\ge0$, summing to one, such that
\begin{equation}\tag{component cancellation}\label{eq:component-cancellation}
q=\sum_{\sigma'}\omega_{\sigma'}v_{j,\sigma'}(y),
\qquad
\alpha\innerL{q}{e_\ell}=\innerL{u}{e_\ell}
\quad\bigl(\ell>\max\{k,j+1\}\bigr).
\end{equation}
As already justified, every ramp used in this combination is active at $y$. Thus \eqref{eq:active-combination} gives $q\in\partial P(y)$, and therefore
\[
g=-u+\alpha q\in\partial f(y).
\]
By \eqref{eq:component-cancellation}, $g$ is orthogonal to every direction normal to $I_{\max\{k,j+1\}}$ at $y$. Hence $g\in\T_y I_{\max\{k,j+1\}}$, as claimed.
\end{proof}

\subsubsection*{Choosing the coefficients and advance-or-stall}
We now choose coefficients satisfying both \eqref{eq:upper-a} and \eqref{eq:lower-a}. Introduce a \emph{tail prefactor} $\xi>0$ and set
\begin{equation}\tag{ramp coefficients}\label{eq:chosen-a}
a_m=\xi\,\frac{\cos^2\theta(\sin\theta)^m}{(\sin\theta)^m+\cos\theta},
\qquad \xi=\frac14,
\qquad \alpha=32\sec\theta.
\end{equation}
This choice completely specifies the hard function $f$ in \eqref{eq:core}. For this coefficient family, the two upper requirements follow whenever $\xi\le1/4$, by retaining respectively $(\sin\theta)^m$ and $\cos\theta$ in the denominator. With $\alpha=32\sec\theta$, the lower requirement is equivalent to $\xi\ge51/256$. Thus the choice
\[
\xi=\frac14=\frac{64}{256}\ge\frac{51}{256}
\]
satisfies all requirements.

\begin{lemma}[Advance-or-stall]\label{lem:stall-advance}
Use the coefficients \eqref{eq:chosen-a}, and define $P$ and $f$ by \eqref{eq:core}. For every $y\in I_k$, $0\le k\le d-2$, the function $f$ has a subgradient $g\in\partial f(y)$ such that
\begin{equation}\tag{advance-or-stall}\label{eq:stall-advance}
\begin{aligned}
&\text{if }P(y)=0,\quad\text{then }g\in\T_yI_{k+1},\\
&\text{if }P(y)>0,\quad\text{then }g\in\T_yI_k.
\end{aligned}
\end{equation}
\end{lemma}
At a query $y\in I_k$ inside every halfspace, the oracle therefore reveals at most one new direction (\emph{advance}). At a query $y\in I_k$ outside some halfspace, it reveals no new direction (\emph{stall}). This stall case gives some justification for assumption A2 in \citet[Appendix~A]{CB23}: querying outside the halfspaces need not help the algorithm discover a new information direction. \Cref{fig:stall-advance} illustrates both cases.
\begin{proof}
If $P(y)=0$, use the ``fresh'' family of ramps $j=k$. Every halfspace $L_{k,\sigma}$ in this family has $I_k$ in its boundary by \eqref{eq:boundary-anchor}, so every pair $(k,\sigma)$ is active at $y$. \Cref{lem:completion} provides a subgradient of $f$ at $y$ tangent to $I_{\max\{k,k+1\}}=I_{k+1}$. Thus at most one new information direction is revealed: an \emph{advance}.

If $P(y)>0$, choose any active pair $(j,\sigma)$, meaning $P(y)=\dist(y,L_{j,\sigma})$. For every $j\ge k$, \eqref{eq:boundary-anchor} and the nesting in \Cref{lem:geometry} give $I_k\subseteq I_j\subseteq\partial L_{j,\sigma}$. Therefore, the ramps with $j\ge k$ have value zero at $y$. Thus, the active pair must have $j<k$, and $y$ lies outside $L_{j,\sigma}$. \Cref{lem:completion} provides a subgradient of $f$ at $y$ tangent to $I_{\max\{k,j+1\}}=I_k$. This is a \emph{stall}: no new information direction is revealed, even though the next query need not equal the current one.
\end{proof}

\begin{figure}[!htbp]
\centering
\includegraphics[width=\linewidth]{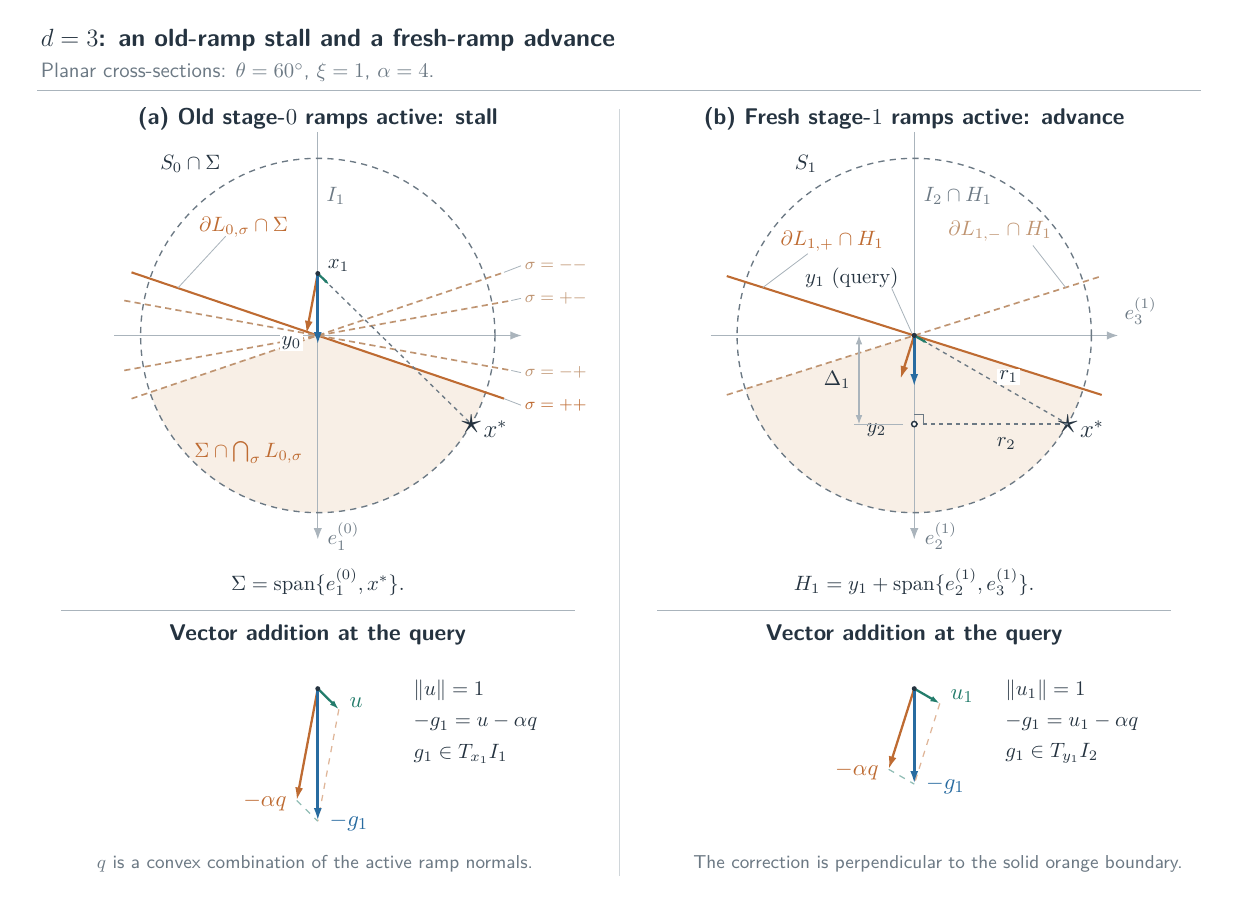}
\caption[Euclidean illustrations of advance-or-stall.]{Euclidean illustrations of advance-or-stall in dimension $d=3$, using planar cross-sections. Left: the query $x_1$ lies on the negative $e_1^{(0)}$ side of $y_0$ and activates all four stage-zero ramps. Their boundary traces in $\Sigma=\spanof(e_1^{(0)},x^*)$ are labelled by their sign pairs. The negative returned subgradient $-g_1$ (blue) is tangent to $I_1$: a \emph{stall}. Right: the query is $x_1=y_1\in I_1$, shown in the normal plane $H_1$. The correction $-\alpha q$ from the fresh ramp $\dist(\cdot,L_{1,+})$ cancels the $e_3^{(1)}$ component of $u_1$, allowing an \emph{advance} into $I_2$.
\par\smallskip
As in \Cref{fig:d2}, green, orange, and blue arrows denote the unit direction toward $x^*$, $-\alpha q$ (see the proof of \Cref{lem:completion}), and their sum $-g_1$, the negative of the subgradient returned to the algorithm. The illustrative parameters are $\theta=60^\circ$, tail prefactor $\xi=1$, rather than the proof value $\xi=1/4$, and $\alpha=4$.}\label{fig:stall-advance}
\end{figure}

\subsection{The lower bound under the geodesic-span assumption}\label{subsec:A1proof}\label{subsec:gspan}
A standard device in Euclidean lower bounds is to first restrict the algorithm to the linear span of its past gradients:
\[
x_0=\xref,\qquad
x_k\in\xref+\spanof(g_0,\ldots,g_{k-1})\quad(k\ge1);
\]
see \citet[Chapter~7]{NY83} and \citet[Section~2.1.2]{Nesterov04}. Equivalently, each query lies in the affine span generated by the preceding queries and their gradient directions. The natural hyperbolic counterpart uses what we call the \emph{geodesic span}, see \eqref{eq:gspan}. We assume that the algorithm chooses each new query in the geodesic span of its previous queries and subgradients:
\begin{equation}\tag{A1}\label{eq:A1}
x_0=\xref,\qquad
x_k\in\gspan(x_0,\ldots,x_{k-1},g_0,\ldots,g_{k-1})\quad(k\ge1).
\end{equation}
Following \citet[Appendix~A]{CB23}, we call this the \emph{geodesic-span assumption}, or A1.

We first prove the $\Omega(\zeta_r/\eps^2)$ lower bound under \eqref{eq:A1}; \Cref{sec:unrestricted} then removes the assumption by a general black-box reduction. Many intrinsic first-order methods already satisfy A1, including the subgradient method with exponential retraction. More generally, all intrinsic first-order algorithms familiar to us that use the exponential map $\exp$ as retraction, and parallel transport or the differential of $\exp$ as transporter \citep[Section~10.5]{Boumal23}, satisfy this assumption.

\begin{theorem}[Tight lower bound under A1]\label{thm:A1}
Let $c=2^{-16}$ and $\eps_0=2^{-9}$. For $r,M>0$ and $0<\eps\le\eps_0$, set
\[
T=\left\lfloor c\frac{\zeta_r}{\eps^2}\right\rfloor,
\qquad d=T.
\]
For every deterministic first-order algorithm on $\HH^d$ satisfying \eqref{eq:A1}, there exist $f\in\calF_{r,M}^d$ and a first-order oracle $\calO_f$ such that running the algorithm on $\calO_f$ produces queries $x_k$ satisfying
\[
f(x_k)-f^*\ge2\eps Mr\qquad(0\le k<T).
\]
The minimizer is unique and lies on $\partial B(\xref,r)$. The conclusion holds even if $f^*=0$ is disclosed to the algorithm in advance.
\end{theorem}
\begin{proof}
Use the function \eqref{eq:core} with coefficients \eqref{eq:chosen-a} and choose $\theta\in(0,\pi/2)$ so that
\begin{equation}\tag{angle choice}\label{eq:theta-main}
\cos\theta=132\eps.
\end{equation}
Since $\eps\le2^{-9}$, we have $\cos\theta\le132/512<1/2$. Also $d=T\ge4$ because $\zeta_r\ge1$. Our parameter choice gives
\begin{equation}\tag{geometric horizon}\label{eq:horizon}
\cos^2\theta\le\frac12,
\qquad d\le\frac{\zeta_r}{2\cos^2\theta},
\end{equation}
where the second inequality follows from $2\cdot132^2<2^{16}$. By \Cref{lem:geometry}, $r_k\ge r/2$ for $0\le k\le d-1$.

By \Cref{lem:margin}, every ramp vanishes at $x^*$. The function $f$ in \eqref{eq:core} is therefore globally g-convex, $(1+\alpha)$-Lipschitz, and satisfies
\begin{equation}\tag{core gap}\label{eq:core-gap}
f^*=f(x^*)=0,\qquad f(x)\ge\dist(x,x^*).
\end{equation}
Its minimizer is unique and belongs to $S_0=\partial B(\xref,r)$. Define $\widetilde f=\frac{M}{1+\alpha}f$. Then $\widetilde f\in\calF_{r,M}^d$, with the same unique minimizer and minimum value $\widetilde f^*=0$.

\paragraph{Specifying the oracle.}
We first specify an oracle $\calO_f$ for $f$. At each $y\in I_{d-2}$, let $m$ be the smallest index with $y\in I_m$. Choose once and for all a subgradient supplied by \Cref{lem:stall-advance} at that index; it is tangent to $I_{m+1}$. At every other point choose any subgradient, again once and for all. This defines a \emph{single} oracle map, so repeated queries receive the same answer. Define $\calO_{\widetilde f}$ by multiplying both the values and the subgradients returned by $\calO_f$ by $M/(1+\alpha)$.

\paragraph{Running the algorithm.}
Run any algorithm satisfying \eqref{eq:A1} with $\calO_{\widetilde f}$, obtaining queries $x_0,x_1,\ldots$. We show inductively that $x_k\in I_k$ for $0\le k\le d-1$. Initially $x_0=\xref=y_0\in I_0$. If $x_k\in I_k$ and $k\le d-2$, the oracle uses an index $m\le k$ and returns a vector tangent to $I_{m+1}\subseteq I_{k+1}$. Every earlier query and answer is also contained in, or tangent to, $I_{k+1}$, by the inductive hypothesis. Assumption \eqref{eq:A1} and the intrinsic characterization \eqref{eq:gspan} therefore give $x_{k+1}\in I_{k+1}$.

Since $y_k$ is the metric projection of $x^*$ onto $I_k$, \Cref{lem:geometry} and \eqref{eq:core-gap} imply
\[
f(x_k)\ge\dist(x_k,x^*)\ge r_k\ge r/2\qquad(0\le k<T).
\]
Therefore,
\[
\widetilde f(x_k)-\widetilde f^*
\ge\frac{Mr}{2(1+32\sec\theta)}
\ge\frac{Mr\cos\theta}{66}=2\eps Mr.
\]
This completes the proof.
\end{proof}

\section{Removing the geodesic-span assumption}\label{sec:unrestricted}
We have established the $\Omega(\zeta_r/\eps^2)$ lower bound under the geodesic-span assumption \eqref{eq:A1}. We now remove that assumption by a general reduction, independent of our reflected-ramp construction.

In Euclidean optimization, there are two implementations of the resisting-oracle principle for removing the linear-span assumption recalled in \Cref{subsec:A1proof}. The first chooses the information subspaces adaptively: a new query adds at most one direction, and the oracle's answer adds at most one more. In this case, the information subspaces grow by at most two dimensions per query, rather than by one as for linear-span algorithms \citep[Chapter~7]{NY83}.

The second approach hides a function $f:\RR^d\to\RR$ in a larger space $\RR^D$ by adversarially choosing an isometric embedding of $\RR^d$ into $\RR^D$; see \citet[Section~3 and Proposition~2]{CDHS20} and \citet[Theorem~12.3]{Li26}. Both approaches can be implemented for the present construction. We use the second, because it yields a black-box reduction that may be useful for other lower bounds in g-convex optimization.

\subsection{A g-convex Lipschitz lifting}\label{subsec:lifting}
Suppose $f:\HH^d\to\RR$ gives a $T$-query lower bound under \eqref{eq:A1}. We seek to place $f$ inside $\HH^D$, with $D=d+T$, and extend it to $F:\HH^D\to\RR$, so that an unrestricted algorithm minimizing $F$ can be \emph{simulated} by a geodesic-span algorithm minimizing $f$. In Euclidean space this is done with $F(x)=f(U^\top x)$, where $U:\RR^d\to\RR^D$ is an isometric embedding chosen adversarially. For hyperbolic space, choose a totally geodesic copy $\mathcal N\subsetneq\HH^D$ of $\HH^d$. After identifying $f$ with its copy on $\mathcal N$, the natural analogue is $F(x)=f(\Proj_{\mathcal N}(x))$, where $\Proj_{\mathcal N}$ is metric projection, as in \Cref{lem:hyperbolic-projection}. However, this composition is \emph{not} g-convex unless $f$ is constant. We first construct a g-convex replacement for any fixed $\mathcal N$; the adversarial choice of embedding $\mathcal N$ is deferred to \Cref{subsec:blackbox}.

We repair $x\mapsto f(\Proj_{\mathcal N}(x))$ by adding $M\log\cosh\dist(x,\mathcal N)$. This choice is suggested by \eqref{eq:general-projection-triangle}: taking logarithms turns the multiplicative hyperbolic Pythagorean identity into the additive formula
\begin{equation}\tag{log-cosh Pythagoras}\label{eq:logcosh-pythagoras}
\begin{split}
\log\cosh\dist(x,z)
={}&\log\cosh\dist(x,\mathcal N)\\
&+\log\cosh\dist(\Proj_{\mathcal N}(x),z),
\qquad z\in\mathcal N.
\end{split}
\end{equation}
Log-cosh distance potentials have appeared in other contexts, for example hyperbolic barycenter constructions \citep[Section~3]{JK24}.

\begin{lemma}[g-convex Lipschitz lifting]\label{lem:lifting}
Let $\mathcal N\subseteq\HH^D$ be totally geodesic, with $\dim\mathcal N\ge1$. Let $f:\mathcal N\to\RR$ be g-convex and $M$-Lipschitz, where $M>0$. Then
\begin{equation}\tag{g-convex lift}\label{eq:lift}
F(x)=f(\Proj_{\mathcal N}(x))+M\log\cosh\dist(x,\mathcal N)
\end{equation}
is g-convex and $M$-Lipschitz on $\HH^D$, and $F|_{\mathcal N}=f$. If $f$ attains its minimum, then $F$ has the same minimum value and exactly the same minimizers, viewed as points of $\mathcal N\subseteq\HH^D$.

For $x\in\HH^D$, put $p=\Proj_{\mathcal N}(x)$ and $\lambda=\cosh\dist(x,\mathcal N)$. Given $g\in\partial f(p)$, a value and subgradient of $F$ are
\begin{equation}\tag{lifted oracle}\label{eq:lifted-oracle}
F(x)=f(p)+M\log \lambda,
\qquad
G=\frac g\lambda+M\left(x-\frac p\lambda\right)\in\partial F(x).
\end{equation}
The vector formula uses the hyperboloid model, with $g$ regarded as an ambient Lorentz vector.
\end{lemma}
Thus one query of a first-order oracle for $f$ suffices to answer one query of $F$.
\begin{proof}
We start with convexity and Lipschitz continuity. Fix a unit-speed ambient geodesic $\gamma$. By \Cref{lem:hyperbolic-projection}, its metric projection onto $\mathcal N$ lies on a geodesic $S\subseteq\mathcal N$ or consists of a single point. In the latter case choose any geodesic $S\subseteq\mathcal N$ through that point, which is possible because $\dim\mathcal N\ge1$.

Parametrize $S$ as in \Cref{lem:busemann-coordinates}, by $\ell(s)=\cosh s\,y+\sinh s\,v$ with $\norm v=1$, and write
\[
\Proj_{\mathcal N}(\gamma(t))=\ell(s(t)),
\qquad \lambda(t)=\cosh\dist(\gamma(t),\mathcal N).
\]
The nearest point in $\mathcal N$ already belongs to $S$, so
\[
\Proj_S(\gamma(t))=\Proj_{\mathcal N}(\gamma(t)),\qquad
\dist(\gamma(t),S)=\dist(\gamma(t),\mathcal N).
\]
The restriction $s\mapsto f(\ell(s))$ is a finite convex $M$-Lipschitz function on $\RR$. It is therefore the supremum of its affine supports $as+b$, with $|a|\le M$. For any such support, \eqref{eq:opposite-busemanns} gives
\begin{equation}\label{eq:lift-busemann-support}
a s(t)+b+M\log \lambda(t)
=b+\tfrac{M+a}{2}b_{y-v}(\gamma(t))+\tfrac{M-a}{2}b_{y+v}(\gamma(t)).
\end{equation}
The two coefficients are nonnegative and sum to $M$. Since the Busemann functions are g-convex and $1$-Lipschitz (\Cref{subsec:busemann}), the right-hand side is convex and $M$-Lipschitz in $t$. Taking the supremum over the affine supports gives exactly $F(\gamma(t))$. A finite-valued supremum of $M$-Lipschitz convex functions is both $M$-Lipschitz and convex. Since $\gamma$ was arbitrary, $F$ is g-convex and $M$-Lipschitz.

The correction in \eqref{eq:lift} is nonnegative and vanishes exactly on $\mathcal N$. This proves the assertions about the restriction and minimizers.

\paragraph{The oracle formula.}
Fix $x$, $p$, $\lambda$, and $g\in\partial f(p)$ as in the statement, and write $\mathcal N=\HH^D\cap V$. Because $g\in V$ and $\innerL{g}{p}=0$, \eqref{eq:metric-projection} gives $\innerL{g}{x}=\innerL{g}{\Pi_Vx}=\lambda\innerL{g}{p}=0$. Also $\innerL{x-p/\lambda}{x}=0$, because $\innerL{x}{x}=-1$ and $-\innerL{p}{x}=\cosh\dist(x,p)=\lambda$ by \eqref{eq:hyperboloid-distance}. Thus the proposed vector $G$ is tangent to $\HH^D$ at $x$.

We verify its global supporting inequality. Fix $z\in\HH^D$ and put $w=\log_x(z)$, $x_t=\exp_x(tw)$, $p_t=\Proj_{\mathcal N}(x_t)$, and $\lambda_t=\cosh\dist(x_t,\mathcal N)$. Since $g$ is a subgradient of $f$ at $p$, for $0<t\le1$,
\[
F(x_t)-F(x)
\ge\inner{g}{\log_p(p_t)}+M\log(\lambda_t/\lambda).
\]
The already proved g-convexity of $F$ also gives $F(x_t)-F(x)\le t(F(z)-F(x))$. Combining these inequalities,
\begin{equation}\label{eq:lift-support-quotient}
F(z)-F(x)\ge
\frac{\inner{g}{\log_p(p_t)}}{t}
+\frac M t\log(\lambda_t/\lambda).
\end{equation}

We use $\dot x_0$, $\dot p_0$, and $\dot \lambda_0$ to denote the derivatives of the curves $x_t$, $p_t$, and $\lambda_t$ at $t=0$. Note that $\dot x_0=w$. By \eqref{eq:metric-projection},
\begin{equation}\label{eq:projection-curves}
\lambda_t^2=-\innerL{\Pi_Vx_t}{\Pi_Vx_t},\qquad
p_t=\frac{\Pi_Vx_t}{\lambda_t}.
\end{equation}
Differentiating the first equality of \eqref{eq:projection-curves}, and using $\Pi_Vx=\lambda p$, gives
\[
2\lambda\dot \lambda_0=-2\innerL{\Pi_Vx}{\Pi_Vw}
=-2\lambda\innerL{p}{w},
\qquad\text{hence}\quad\dot \lambda_0=-\innerL{p}{w}.
\]
Differentiating the second equality of \eqref{eq:projection-curves} yields
\begin{align*}
\innerL{g}{\dot p_0}
&=\left\langle g,\frac{\Pi_Vw}{\lambda}-\frac{\dot \lambda_0}{\lambda^2}\Pi_Vx\right\rangle_{\mathrm L}\\
&=\frac{\innerL{g}{\Pi_Vw}}\lambda-\frac{\dot \lambda_0}{\lambda}\innerL{g}{p}
=\frac{\innerL{g}{w}}\lambda,
\end{align*}
where $g\in V$ and $\innerL{g}{p}=0$ justify the last equality. The differential of $\log_p$ at $p$ is the identity, since it is inverse to the differential of $\exp_p$ at zero. Thus $\log_p(p_t)/t\to\dot p_0$. Letting $t\downarrow0$ in \eqref{eq:lift-support-quotient} therefore gives
\[
F(z)-F(x)\ge\frac{\innerL{g}{w}}\lambda-\frac M\lambda\innerL{p}{w}
=\innerL{G}{w},
\]
where the last equality uses $\innerL{x}{w}=0$. This is \eqref{eq:subgradient} for $G$, proving that $G$ is a subgradient of $F$ at $x$.
\end{proof}

\subsection{The black-box reduction}\label{subsec:blackbox}
We now give the black-box reduction, following the Euclidean hidden-subspace argument of \citet[Section~3 and Proposition~2]{CDHS20} and \citet[Theorem~12.3]{Li26}.

\begin{theorem}[Black-box removal of the geodesic-span assumption]\label{thm:blackbox}
Let $d\ge2$ and $T\ge1$ be integers, and let $r,M,\Delta>0$. Suppose that for every deterministic first-order algorithm satisfying \eqref{eq:A1} on $\HH^d$, there exist $f\in\calF_{r,M}^d$ and a first-order oracle for $f$ such that the algorithm's first $T$ queries have function gap $f(\cdot)-f^*$ at least $\Delta$.

Then, for every deterministic first-order algorithm on $\HH^{d+T}$, there exist $F\in\calF_{r,M}^{d+T}$ and a first-order oracle for $F$ such that the algorithm's first $T$ queries have function gap at least $\Delta$.
\end{theorem}
\begin{proof}
Put $D=d+T$ and fix a deterministic first-order algorithm $\calA=\calA^D$ on $\HH^D$. We inductively define a deterministic algorithm $\mathcal B=\mathcal B^d$ on $\HH^d$ by specifying its first query and its response to successive answers from an arbitrary oracle $\calO_f$. Algorithm $\mathcal B$ runs $\calA$ as a subroutine, simulating its interaction with a first-order oracle. We call this procedure a \emph{simulation}: after selecting a hard function for $\mathcal B$, we construct a single function $F$ and a valid oracle $\calO_F$ for which the simulated queries and answers are exactly the genuine interaction of $\calA$ with $\calO_F$. The rule defining $\mathcal B$ uses no prior knowledge of $f$; only after defining $\mathcal B$ do we choose its hard function and oracle from the hypothesis.

While defining $\mathcal B$, we inductively construct information submanifolds $I_k\subseteq\HH^d$ and isometric embeddings $\iota_k:I_k\to\HH^D$. We also construct the queries $x_k$ of $\calA$, the queries $y_k$ of $\mathcal B$, and their associated projections and subgradients. At each step we verify the invariants below, \ref{eq:simulation-span}--\ref{eq:simulation-subgradients}, which in particular imply that $\mathcal B$ satisfies \eqref{eq:A1}. The $I_k$ here are constructed during the simulation, rather than prescribed as in \Cref{sec:geometry}; we use the same notation because they play the same role as information submanifolds.

\paragraph{Initialization.}
To start the induction, take both reference points to be $e_0$ in their respective hyperboloid models, without loss of generality:
\[
\xref^d=e_0\in\RR^{d+1},\qquad \xref^D=e_0\in\RR^{D+1}.
\]
Set $I_0=\{\xref^d\}\subseteq\HH^d$, and let $\iota_0:I_0\to\HH^D$ map $\xref^d$ to $\xref^D$. With an initially empty simulated transcript, $\calA$ determines its first query $x_0\in\HH^D$. Put
\[
p_0=\xref^D,\qquad y_0=\xref^d,\qquad
\lambda_0=\cosh\dist(x_0,p_0).
\]
The first query of $\mathcal B$ is $y_0=\xref^d$.

\paragraph{Inductive hypotheses.}
Consider the beginning of round $0\le k<T$. Algorithm $\mathcal B$ has received the oracle answers to its first $k$ queries $y_0,\ldots,y_{k-1}$. These answers have been used to construct $I_k$ and $\iota_k$, and to simulate the first $k$ oracle interactions of $\calA$. Thus the next query $x_k$ of $\calA$ has already been determined. The next query of $\mathcal B$ is then given by
\[
p_k=\Proj_{\iota_k(I_k)}(x_k),\qquad
 y_k=\iota_k^{-1}(p_k),\qquad
 \lambda_k=\cosh\dist(x_k,p_k).
\]
For $j<k$, write $(f(y_j),g_j)$ for the answer of $\calO_f$ at $y_j$, and $\widetilde g_j$ for the embedded subgradient used in the simulated reply to $\calA$. We inductively maintain the following five properties for $0\le k\le T$, with $(\mathrm{IH}2)_k$ required only when $k<T$:
\begin{equation}\tag*{$(\mathrm{IH}1)_k$}\label{eq:simulation-span}
\begin{gathered}
I_k=\gspan(\xref^d,y_0,\ldots,y_{k-1},g_0,\ldots,g_{k-1}),\\
\dim I_k\le\min\{k,d\};
\end{gathered}
\end{equation}
\begin{equation}\tag*{$(\mathrm{IH}2)_k$}\label{eq:simulation-query}
y_k\in I_k\qquad(k<T);
\end{equation}
\begin{equation}\tag*{$(\mathrm{IH}3)_k$}\label{eq:simulation-extensions}
\iota_k|_{I_j}=\iota_j\qquad(0\le j<k);
\end{equation}
\begin{equation}\tag*{$(\mathrm{IH}4)_k$}\label{eq:simulation-projections}
\Proj_{\iota_k(I_k)}(x_j)=p_j=\iota_k(y_j)
\qquad(0\le j\le k,\ j<T);
\end{equation}
\begin{equation}\tag*{$(\mathrm{IH}5)_k$}\label{eq:simulation-subgradients}
d\iota_k(y_j)[g_j]=\widetilde g_j\qquad(0\le j<k).
\end{equation}
Every $\iota_k$ is a totally geodesic isometric embedding fixing the reference points; the construction below realizes it as the restriction of a linear map that preserves the Lorentz inner product. Conditions $(\mathrm{IH}4)_k$ and $(\mathrm{IH}5)_k$ ensure that enlarging the embedding does not change any previously used projection or embedded subgradient. At $k=0$, $(\mathrm{IH}1)_0$ is the definition of $I_0$, and $(\mathrm{IH}2)_0$ holds because $y_0=\xref^d\in I_0$. Condition $(\mathrm{IH}4)_0$ follows from
\[
\Proj_{\iota_0(I_0)}(x_0)=\xref^D=p_0=\iota_0(y_0).
\]
Conditions $(\mathrm{IH}3)_0$ and $(\mathrm{IH}5)_0$ are vacuous. We now define $\mathcal B$'s response to its oracle answer at round $k<T$ and verify the invariants $(\mathrm{IH}1)_{k+1}$--$(\mathrm{IH}5)_{k+1}$.

\paragraph{Receiving an oracle answer and enlarging the information submanifold.}
At its already determined query $y_k$, algorithm $\mathcal B$ receives $(f(y_k),g_k)$ from $\calO_f$. Set
\[
I_{k+1}=\gspan(I_k,y_k,g_k).
\]
By $(\mathrm{IH}2)_k$, $y_k\in I_k$, so adjoining $g_k$ adds at most one dimension. Together with $(\mathrm{IH}1)_k$, this proves $(\mathrm{IH}1)_{k+1}$. We next extend the embedding to $I_{k+1}$ so that $g_k$ can be mapped into the tangent space $\T_{p_k}\HH^D$ and used to answer $\calA$'s query $x_k$.

\paragraph{Extending the partial embedding.}
If $I_{k+1}=I_k$, set $\iota_{k+1}=\iota_k$. Otherwise, $\dim I_{k+1}=\dim I_k+1$, so $\dim I_k<d$. Choose a unit vector $v_k\in\T_{\xref^d}I_{k+1}$ orthogonal to $\T_{\xref^d}I_k$; this choice is unique up to sign. Define
\[
J_k=\gspan\bigl(\iota_k(I_k),x_0,\ldots,x_k\bigr).
\]
There is room to choose a unit vector $w_k\in\T_{\xref^D}\HH^D$ orthogonal to $\T_{\xref^D}J_k$, because
\[
\dim J_k\le\dim I_k+k+1<d+k+1\le d+T=D.
\]
The choices of $v_k$ and $w_k$ are made by fixed deterministic rules, specified below. We seek an extension satisfying
\[
\iota_{k+1}|_{I_k}=\iota_k,\qquad
d\iota_{k+1}(\xref^d)[v_k]=w_k.
\]

To construct it, put $E_k=\RR\xref^d\oplus\T_{\xref^d}I_k$, so that $I_k=\HH^d\cap E_k$. The map $\iota_k$ is the restriction of a linear map $\overline\iota_k:E_k\to\RR^{D+1}$ that preserves the Lorentz inner product and maps $\xref^d$ to $\xref^D$.\footnote{Indeed, for $a\in\RR$ and $u\in\T_{\xref^d}I_k$, set $\overline\iota_k(a\xref^d+u)=a\xref^D+d\iota_k(\xref^d)[u]$.} Let $E_{k+1}=E_k\oplus\RR v_k$, so that $I_{k+1}=E_{k+1}\cap\HH^d$. Define the linear map $\overline\iota_{k+1}:E_{k+1}\to\RR^{D+1}$~by
\begin{equation}\label{eq:linear-embedding-extension}
\overline\iota_{k+1}(z+bv_k)=\overline\iota_k(z)+bw_k
\qquad(z\in E_k,\ b\in\RR).
\end{equation}
We verify that this map preserves the Lorentz inner product and the positive sheet of the hyperboloid.

First, $v_k$ is Lorentz-orthogonal to $E_k$: it is tangent at $\xref^d$, and was chosen orthogonal to $\T_{\xref^d}I_k$. Similarly, $w_k$ is Lorentz-orthogonal to $\overline\iota_k(E_k)$: it is tangent at $\xref^D$, and was chosen orthogonal to $\T_{\xref^D}J_k$, which contains $\T_{\xref^D}\iota_k(I_k)$. Both vectors have Lorentz norm one. Consequently, for $z,z'\in E_k$ and $b,b'\in\RR$,
\[
\begin{aligned}
\innerL{\overline\iota_{k+1}(z+bv_k)}{\overline\iota_{k+1}(z'+b'v_k)}
&=\innerL{\overline\iota_k(z)}{\overline\iota_k(z')}+bb'\\
&=\innerL{z}{z'}+bb'\\
&=\innerL{z+bv_k}{z'+b'v_k}.
\end{aligned}
\]
Thus $\overline\iota_{k+1}$ preserves the Lorentz inner product.

Next, let $h\in I_{k+1}$. Preservation of the Lorentz inner product gives
\[
\innerL{\overline\iota_{k+1}(h)}{\overline\iota_{k+1}(h)}
=\innerL{h}{h}=-1.
\]
Moreover, since $\overline\iota_{k+1}(\xref^d)=\xref^D=e_0$,
\[
-\innerL{\overline\iota_{k+1}(h)}{\xref^D}
=-\innerL{h}{\xref^d}>0.
\]
Hence $\overline\iota_{k+1}(h)$ belongs to the positive sheet $\HH^D$.

The restriction of $\overline\iota_{k+1}$ to $I_{k+1}$ is therefore the required isometric embedding $\iota_{k+1}$, with totally geodesic image $\HH^D\cap\overline\iota_{k+1}(E_{k+1})$ by \eqref{eq:gspan}. The extension agrees with $\iota_k$ on $I_k$ by definition. Its differential sends $v_k$ to $w_k$, since the derivative of the ambient linear map is the map itself. Thus $(\mathrm{IH}3)_{k+1}$ follows from $(\mathrm{IH}3)_k$.

\paragraph{Preserving the projections and subgradients.}
Suppose $I_k\subsetneq I_{k+1}$. As shown above, $w_k$ is Lorentz-orthogonal to $\overline\iota_k(E_k)$. Moreover, since $x_j\in J_k$ for $0\le j\le k$, our choice of $w_k$ gives $\innerL{x_j}{w_k}=0$ for $0\le j\le k$. By \eqref{eq:linear-embedding-extension},
\[
\overline\iota_{k+1}(E_{k+1})=\overline\iota_k(E_k)\oplus\RR w_k.
\]
Consequently,
\begin{equation}\label{eq:projection-preservation}
\Pi_{\overline\iota_{k+1}(E_{k+1})}x_j
=\Pi_{\overline\iota_k(E_k)}x_j+\innerL{x_j}{w_k}w_k
=\Pi_{\overline\iota_k(E_k)}x_j
\qquad(0\le j\le k).
\end{equation}
Therefore, by \eqref{eq:metric-projection} and $(\mathrm{IH}4)_k$,
\[
\Proj_{\iota_{k+1}(I_{k+1})}(x_j)
=\Proj_{\iota_k(I_k)}(x_j)=p_j
\qquad(0\le j\le k).
\]
Since $y_j\in I_k$ for every $j\le k$ and $\iota_{k+1}|_{I_k}=\iota_k$, we also have
\[
\iota_{k+1}(y_j)=\iota_k(y_j)=p_j
\qquad(0\le j\le k).
\]
This establishes the entries of $(\mathrm{IH}4)_{k+1}$ indexed by $j\le k$; when $k<T-1$, the new entry $j=k+1$ is established below. If $I_{k+1}=I_k$, the same conclusions hold trivially.

For $j<k$, $(\mathrm{IH}1)_k$ gives $g_j\in\T_{y_j}I_k$. Since $\iota_{k+1}|_{I_k}=\iota_k$, we have
\[
d\iota_{k+1}(y_j)[g_j]=d\iota_k(y_j)[g_j]=\widetilde g_j,
\]
where the last equality follows from $(\mathrm{IH}5)_k$.
For the new index $j=k$, define
\[
\widetilde g_k=d\iota_{k+1}(y_k)[g_k]\in\T_{p_k}\HH^D.
\]
This is well-defined because $g_k\in\T_{y_k}I_{k+1}$, and establishes the remaining entry of $(\mathrm{IH}5)_{k+1}$.

\paragraph{Answering $\calA$ and determining the next query of $\mathcal B$.}
Following \eqref{eq:lifted-oracle}, we respond to $\calA$'s query $x_k$ with the pair
\begin{equation}\tag{simulated answer}\label{eq:simulated-answer}
\left(f(y_k)+M\log\lambda_k,\quad
\frac{\widetilde g_k}{\lambda_k}
+M\left(x_k-\frac{p_k}{\lambda_k}\right)\right).
\end{equation}
If $k<T-1$, this reply determines the next query $x_{k+1}$ of $\calA$. Define
\[
\begin{aligned}
p_{k+1}&=\Proj_{\iota_{k+1}(I_{k+1})}(x_{k+1}),\\
y_{k+1}&=\iota_{k+1}^{-1}(p_{k+1}),\qquad
\lambda_{k+1}=\cosh\dist(x_{k+1},p_{k+1}).
\end{aligned}
\]
Algorithm $\mathcal B$ issues $y_{k+1}$ as its next query. By construction $y_{k+1}\in I_{k+1}$, proving $(\mathrm{IH}2)_{k+1}$. The definitions of $p_{k+1}$ and $y_{k+1}$ establish the new $j=k+1$ entry of $(\mathrm{IH}4)_{k+1}$. This completes the inductive step. At $k=T-1$, the simulation stops after sending \eqref{eq:simulated-answer}; no further query is needed. At this point, $(\mathrm{IH}1)_T$, $(\mathrm{IH}3)_T$, $(\mathrm{IH}4)_T$, and $(\mathrm{IH}5)_T$ have been established.

\paragraph{End of the simulation: determinism and the geodesic-span assumption.}
To make $\mathcal B$ deterministic, we fix the choices of $v_k$ and $w_k$ as follows: project the standard basis vectors $e_1,e_2,\ldots$ in order onto the required subspace and normalize the first nonzero projection. With these rules, every state of the simulation is determined by $\calA$, the known parameters, and the oracle answers already received by $\mathcal B$.

After the first $T$ queries, define $\mathcal B$ to repeat $\xref^d$. This completes the definition of the deterministic algorithm $\mathcal B$. Its first query is $y_0=\xref^d$, and $(\mathrm{IH}1)_k$ and $(\mathrm{IH}2)_k$ imply \eqref{eq:A1} for the first $T$ queries. The subsequent queries also satisfy \eqref{eq:A1}, so $\mathcal B$ satisfies the geodesic-span assumption for all rounds. Each of the first $T$ rounds of $\mathcal B$ makes exactly one query to $\calO_f$ and produces exactly one simulated oracle answer for $\calA$. We now choose a hard instance for this algorithm.

\paragraph{Selecting the hard instance and completing the embedding.}
Now apply the hypothesis of \Cref{thm:blackbox} to $\mathcal B$: choose $f\in\calF_{r,M}^d$ and a fixed oracle $\calO_f$ for which $\mathcal B$'s first $T$ queries have $f(y_k)-f^*\ge\Delta$. Run the preceding simulation with this oracle. Its final state contains an embedding $\iota_T:I_T\to\HH^D$, which need not yet be defined on all of $\HH^d$.

Write $s=\dim I_T$ and set
\[
J_T=\gspan\bigl(\iota_T(I_T),x_0,\ldots,x_{T-1}\bigr).
\]
This submanifold has dimension at most $s+T$, so its normal space at $\xref^D$ in $\HH^D$ has dimension at least $D-s-T=d-s$. Following \eqref{eq:linear-embedding-extension}, let $\iota:\HH^d\to\HH^D$ be the isometric embedding extending $\iota_T$ whose differential at $\xref^d$ maps an orthonormal basis of the complement of $\T_{\xref^d}I_T$ in $\T_{\xref^d}\HH^d$ to $d-s$ orthonormal vectors in this normal space of $J_T$. Since the added directions are orthogonal to all previous queries, the projection-preservation argument of \eqref{eq:projection-preservation} applies to this extension as well. Moreover, the embedded subgradients are unchanged because $\iota$ extends $\iota_T$. Thus, by $(\mathrm{IH}4)_T$ and $(\mathrm{IH}5)_T$, we obtain
\begin{equation}\label{eq:final-projection}
\begin{aligned}
\Proj_{\mathcal N}(x_k)&=p_k=\iota(y_k),\\
d\iota(y_k)[g_k]&=\widetilde g_k,
\end{aligned}
\qquad 0\le k<T,
\qquad \mathcal N=\iota(\HH^d).
\end{equation}

\paragraph{One function and one oracle realize the simulation.}
Apply \Cref{lem:lifting} to $f\circ\iota^{-1}$ on $\mathcal N$, obtaining
\[
F(x)=f\bigl(\iota^{-1}(\Proj_{\mathcal N}(x))\bigr)
+M\log\cosh\dist(x,\mathcal N).
\]
Define $\calO_F$ on all of $\HH^D$ by querying $\calO_f$ at $y=\iota^{-1}(\Proj_{\mathcal N}(x))$, transporting the returned subgradient $g$ by $d\iota(y)[g]$, and applying \eqref{eq:lifted-oracle}. Since $\iota$ is an isometric embedding, $d\iota(y)[g]$ is a subgradient of $f\circ\iota^{-1}$ at $\iota(y)$, intrinsically on $\mathcal N$. Thus \Cref{lem:lifting} shows that this defines a valid first-order oracle for $F$.

By \eqref{eq:final-projection}, the oracle's answer at each $x_k$ is exactly \eqref{eq:simulated-answer}. Both executions of $\calA$ begin with the same query $x_0$. Since they receive identical oracle answers at corresponding queries, determinism of $\calA$ implies inductively that their first $T$ queries coincide. Thus running $\calA$ on $\calO_F$ reproduces the simulated transcript. Since $F^*=f^*$ by \Cref{lem:lifting} and $\lambda_k\ge1$, for $0\le k<T$ we obtain
\begin{equation}\label{eq:lift-gap}
F(x_k)-F^*=f(y_k)-f^*+M\log \lambda_k
\ge f(y_k)-f^*\ge\Delta.
\end{equation}
By \Cref{lem:lifting}, $F$ is g-convex and $M$-Lipschitz, and its minimizers are exactly the images under $\iota$ of the minimizers of $f$. Since
\[
\dist(\xref^D,\iota(y))=\dist(\xref^d,y)
\qquad(y\in\HH^d),
\]
one such minimizer lies in $B(\xref^D,r)$. Thus $F\in\calF_{r,M}^{d+T}$, completing the reduction. If the lower-bound hypothesis holds when $f^*$ is disclosed in advance, the same conclusion holds when $F^*$ is disclosed: the simulation passes the same minimum value to $\calA$, since $F^*=f^*$.
\end{proof}

\paragraph{Proof of \Cref{thm:main}.}\label{subsec:main-finish}
Take $T$ from \eqref{eq:main-budget}. \Cref{thm:A1}, applied in dimension $d=T$, gives the lower bound $2\eps Mr$ for the first $T$ queries of every geodesic-span algorithm. \Cref{thm:blackbox} transfers it to arbitrary deterministic algorithms on $\HH^{2T}$. The minimizer of $f:\HH^T\to\RR$ provided by \Cref{thm:A1} is unique and lies on $\partial B(\xref,r)$, so the lifted minimizer has the same properties. This proves \Cref{thm:main}.

\section{Perspectives}\label{sec:scope}
We suspect that the present lower-bound construction can be used to resolve several other complexity questions. An appropriate smoothing argument should extend our construction to the multiplicative lower bound $\Omega(\sqrt{\zeta_r/\eps})$ for globally $L$-smooth g-convex optimization, i.e., the smooth problem P3 of \citet{CB23}. The intended approach is analogous to the smoothing arguments of \citet{GN15} in Euclidean optimization and \citet[Section~5]{CB23} in g-convex optimization. Such an extension could, in turn, yield a multiplicative lower bound $\Omega(\sqrt{\zeta_r\kappa})$ for smooth strongly g-convex optimization, where $\kappa=L/\mu$ and $\mu$ is the strong-convexity constant, via the reduction of \citet[Appendix~I, Proposition~57]{CB23}. We also expect our $\Omega(\zeta_r/\eps^2)$ lower bound to extend to randomized first-order algorithms, but do not establish that here.

In contrast, the complexity of Lipschitz g-convex optimization in \emph{fixed} dimension---problem P1 of \citet{CB23}---appears to require ideas quite different from those underlying the present construction. Our construction provides a hard query horizon proportional to the dimension $d$, rather than $\zeta_r d$, and therefore does not resolve the conjectured $\zeta_r d$ dependence for problem P1.

\Needspace{11\baselineskip}
\section*{Acknowledgments}
\addcontentsline{toc}{section}{Acknowledgments}
The author thanks Nicolas Boumal for a discussion of the proof and for helpful comments on the abstract.

\paragraph{AI.}
The author used GPT-6 Pro in preparing this work. The geometric setup and original hard-function ansatz are due to \citet[Appendix~A]{CB23}; see \Cref{sec:geometry,subsec:ansatz}. The use of reflections and the advance-or-stall scheme (\Cref{subsec:ansatz,subsec:completion}) were developed by the author before any AI assistance. The tuned normals $W_{j,\sigma}$, given by \eqref{eq:W-reflected} and \eqref{eq:chosen-a}, and their analysis (most importantly \Cref{lem:oldramp}) are due to GPT-6 Pro. Earlier attempts with GPT-5.6 Pro did not yield such a construction. GPT-6 Pro also assisted with the black-box reduction in \Cref{sec:unrestricted}, with writing, and with the generation of \Cref{fig:overview,fig:d2,fig:stall-advance}. The author assumes responsibility for all content.

\renewcommand{\bibsection}{\section*{\refname}\addcontentsline{toc}{section}{\refname}}
\bibliographystyle{plainnat}
\bibliography{tight_lower_bounds_hyperbolic}

@inproceedings{CB23,
  author = {Criscitiello, Christopher and Boumal, Nicolas},
  title = {Curvature and complexity: Better lower bounds for geodesically convex optimization},
  booktitle = {Proceedings of the Thirty Sixth Conference on Learning Theory},
  series = {Proceedings of Machine Learning Research},
  volume = {195}, pages = {2969--3013}, year = {2023}, publisher = {PMLR},
  url = {https://proceedings.mlr.press/v195/criscitiello23a.html}
}

@inproceedings{CB22,
  author = {Criscitiello, Christopher and Boumal, Nicolas},
  title = {Negative curvature obstructs acceleration for strongly geodesically convex optimization, even with exact first-order oracles},
  booktitle = {Proceedings of the Thirty Fifth Conference on Learning Theory},
  series = {Proceedings of Machine Learning Research},
  volume = {178}, pages = {496--542}, year = {2022}, publisher = {PMLR},
  url = {https://proceedings.mlr.press/v178/criscitiello22a.html}
}

@inproceedings{ZS16,
  author = {Zhang, Hongyi and Sra, Suvrit},
  title = {First-order methods for geodesically convex optimization},
  booktitle = {Proceedings of the Twenty Ninth Conference on Learning Theory},
  series = {Proceedings of Machine Learning Research},
  volume = {49}, pages = {1617--1638}, year = {2016}, publisher = {PMLR},
  url = {https://proceedings.mlr.press/v49/zhang16b.html}
}

@inproceedings{HM21,
  author = {Hamilton, Linus and Moitra, Ankur},
  title = {A no-go theorem for robust acceleration in the hyperbolic plane},
  booktitle = {Advances in Neural Information Processing Systems},
  volume = {34}, year = {2021},
  url = {https://proceedings.neurips.cc/paper/2021/hash/201d546992726352471cfea6b0df0a48-Abstract.html}
}

@article{Karcher77,
  author = {Karcher, Hermann},
  title = {Riemannian center of mass and mollifier smoothing},
  journal = {Communications on Pure and Applied Mathematics},
  volume = {30}, number = {5}, pages = {509--541}, year = {1977},
  doi = {10.1002/cpa.3160300502}
}

@article{Fletcher09,
  author = {Fletcher, P. Thomas and Venkatasubramanian, Suresh and Joshi, Sarang},
  title = {The geometric median on {Riemannian} manifolds with application to robust atlas estimation},
  journal = {NeuroImage}, volume = {45}, number = {1, Supplement},
  pages = {S143--S152}, year = {2009}, doi = {10.1016/j.neuroimage.2008.10.052}
}

@inproceedings{BFGOWW19,
  author = {B{\"u}rgisser, Peter and Franks, Cole and Garg, Ankit and Oliveira, Rafael and Walter, Michael and Wigderson, Avi},
  title = {Towards a theory of non-commutative optimization: Geodesic first and second order methods for moment maps and polytopes},
  booktitle = {2019 IEEE 60th Annual Symposium on Foundations of Computer Science},
  pages = {845--861}, year = {2019}, publisher = {IEEE},
  doi = {10.1109/FOCS.2019.00055}
}

@book{BH99,
  author = {Bridson, Martin R. and Haefliger, Andr{\'e}},
  title = {Metric Spaces of Non-Positive Curvature},
  series = {Grundlehren der mathematischen Wissenschaften}, volume = {319},
  publisher = {Springer}, year = {1999}, doi = {10.1007/978-3-662-12494-9}
}

@book{Ratcliffe19,
  author = {Ratcliffe, John G.}, title = {Foundations of Hyperbolic Manifolds},
  edition = {3rd}, series = {Graduate Texts in Mathematics}, volume = {149},
  publisher = {Springer}, year = {2019}, doi = {10.1007/978-3-030-31597-9}
}

@book{Boumal23,
  author = {Boumal, Nicolas},
  title = {An Introduction to Optimization on Smooth Manifolds},
  publisher = {Cambridge University Press}, year = {2023},
  doi = {10.1017/9781009166164}
}

@book{Nesterov04,
  author = {Nesterov, Yurii},
  title = {Introductory Lectures on Convex Optimization: A Basic Course},
  series = {Applied Optimization}, volume = {87}, publisher = {Kluwer Academic Publishers},
  year = {2004}
}

@book{NY83,
  author = {Nemirovsky, A. S. and Yudin, D. B.},
  title = {Problem Complexity and Method Efficiency in Optimization},
  publisher = {Wiley}, year = {1983}
}

@article{AF06,
  author = {Azagra, Daniel and Ferrera, Juan},
  title = {Inf-convolution and regularization of convex functions on {Riemannian} manifolds of nonpositive curvature},
  journal = {Revista Matem{\'a}tica Complutense}, volume = {19}, number = {2},
  pages = {323--345}, year = {2006}, doi = {10.5209/rev_REMA.2006.v19.n2.16592}
}

@article{Rusciano20,
  author = {Rusciano, Alexander},
  title = {A {Riemannian} Corollary of {Helly}'s Theorem},
  journal = {Journal of Convex Analysis},
  volume = {27}, number = {4}, pages = {1261--1275}, year = {2020},
  doi = {10.68381/jca27067},
  url = {https://journalofconvexanalysis.com/articles/jca27067}
}

@inproceedings{CMB23,
  author = {Criscitiello, Christopher and Mart{\'i}nez-Rubio, David and Boumal, Nicolas},
  title = {Open Problem: Polynomial linearly-convergent method for g-convex optimization?},
  booktitle = {Proceedings of Thirty Sixth Conference on Learning Theory},
  series = {Proceedings of Machine Learning Research}, volume = {195},
  pages = {5950--5956}, year = {2023}, publisher = {PMLR},
  url = {https://proceedings.mlr.press/v195/criscitiello23b.html}
}

@book{Udriste94,
  author = {Udri\c{s}te, Constantin},
  title = {Convex Functions and Optimization Methods on {Riemannian} Manifolds},
  series = {Mathematics and Its Applications}, volume = {297},
  publisher = {Kluwer Academic Publishers}, year = {1994},
  doi = {10.1007/978-94-015-8390-9}
}

@book{Lee18,
  author = {Lee, John M.},
  title = {Introduction to {Riemannian} Manifolds},
  series = {Graduate Texts in Mathematics}, volume = {176},
  edition = {2nd}, publisher = {Springer}, year = {2018},
  doi = {10.1007/978-3-319-91755-9}
}

@incollection{Chen00,
  author = {Chen, Bang-Yen}, title = {Riemannian submanifolds},
  booktitle = {Handbook of Differential Geometry}, volume = {1},
  editor = {Dillen, Franki J. E. and Verstraelen, Leopold C. A.},
  pages = {187--418}, publisher = {North-Holland}, year = {2000},
  doi = {10.1016/S1874-5741(00)80006-0}
}

@article{GN15,
  author = {Guzm{\'a}n, Crist{\'o}bal and Nemirovski, Arkadi},
  title = {On lower complexity bounds for large-scale smooth convex optimization},
  journal = {Journal of Complexity}, volume = {31}, number = {1},
  pages = {1--14}, year = {2015}, doi = {10.1016/j.jco.2014.08.003}
}

@book{bacak2014hadamard,
  author = {Ba\v{c}{\'a}k, Miroslav},
  title = {Convex Analysis and Optimization in {Hadamard} Spaces},
  series = {De Gruyter Series in Nonlinear Analysis and Applications},
  volume = {22}, publisher = {De Gruyter}, year = {2014},
  doi = {10.1515/9783110361629}
}

@article{auderset2005angular,
  author = {Auderset, Claude and Mazza, Christian and Ruh, Ernst A.},
  title = {Angular {Gaussian} and {Cauchy} estimation},
  journal = {Journal of Multivariate Analysis},
  volume = {93}, number = {1}, pages = {180--197}, year = {2005}
}

@article{wiesel2012gconvexity,
  author = {Wiesel, Ami},
  title = {Geodesic convexity and covariance estimation},
  journal = {IEEE Transactions on Signal Processing},
  volume = {60}, number = {12}, pages = {6182--6189}, year = {2012},
  doi = {10.1109/TSP.2012.2218241}
}

@article{wiesel2015gconvexity,
  author = {Wiesel, Ami and Zhang, Teng},
  title = {Structured robust covariance estimation},
  journal = {Foundations and Trends in Signal Processing},
  volume = {8}, number = {3}, pages = {127--216}, year = {2014},
  doi = {10.1561/2000000053}
}

@misc{ciobotaru2020geometrical,
  author = {Ciobotaru, Corina and Mazza, Christian},
  title = {Geometrical and statistical properties of {M}-estimates of scatter on {Grassmann} manifolds},
  year = {2020}, howpublished = {arXiv:1812.11605v2},
  url = {https://arxiv.org/abs/1812.11605v2}
}

@inproceedings{franksmoitra2020,
  author = {Franks, William Cole and Moitra, Ankur},
  title = {Rigorous guarantees for {Tyler}'s {M}-estimator via quantum expansion},
  booktitle = {Proceedings of the Thirty Third Conference on Learning Theory},
  series = {Proceedings of Machine Learning Research},
  volume = {125}, pages = {1601--1632}, year = {2020}, publisher = {PMLR},
  url = {https://proceedings.mlr.press/v125/franks20a.html}
}

@article{tang2021integrated,
  author = {Tang, Tiffany M. and Allen, Genevera I.},
  title = {Integrated principal components analysis},
  journal = {Journal of Machine Learning Research},
  volume = {22}, number = {198}, pages = {1--71}, year = {2021},
  url = {https://jmlr.org/papers/v22/20-084.html}
}

@article{amendola2021,
  author = {Am{\'e}ndola, Carlos and Kohn, Kathl{\'e}n and Reichenbach, Philipp and Seigal, Anna},
  title = {Invariant theory and scaling algorithms for maximum likelihood estimation},
  journal = {SIAM Journal on Applied Algebra and Geometry},
  volume = {5}, number = {2}, pages = {304--337}, year = {2021},
  doi = {10.1137/20M1328932}
}

@article{franks2021neartyler,
  author = {Franks, Cole and Oliveira, Rafael and Ramachandran, Akshay and Walter, Michael},
  title = {Near optimal sample complexity for matrix and tensor normal models via geodesic convexity},
  journal = {The Annals of Statistics},
  volume = {54}, number = {1}, pages = {93--119}, year = {2026},
  doi = {10.1214/25-AOS2539}
}

@inproceedings{tensorscaling2018,
  author = {B{\"u}rgisser, Peter and Franks, Cole and Garg, Ankit and Oliveira, Rafael and Walter, Michael and Wigderson, Avi},
  title = {Efficient algorithms for tensor scaling, quantum marginals, and moment polytopes},
  booktitle = {2018 IEEE 59th Annual Symposium on Foundations of Computer Science (FOCS)},
  pages = {883--897}, year = {2018}, publisher = {IEEE},
  doi = {10.1109/FOCS.2018.00088}
}

@inproceedings{sravishnoibracsampliebconstant,
  author = {Sra, Suvrit and Vishnoi, Nisheeth K. and Yildiz, Ozan},
  title = {On geodesically convex formulations for the {Brascamp--Lieb} constant},
  booktitle = {Approximation, Randomization, and Combinatorial Optimization. Algorithms and Techniques (APPROX/RANDOM 2018)},
  series = {Leibniz International Proceedings in Informatics (LIPIcs)},
  volume = {116}, pages = {25:1--25:15}, year = {2018},
  publisher = {Schloss Dagstuhl--Leibniz-Zentrum f{\"u}r Informatik},
  doi = {10.4230/LIPIcs.APPROX-RANDOM.2018.25}
}

@book{Bhatia,
  author = {Bhatia, Rajendra},
  title = {Positive Definite Matrices},
  publisher = {Princeton University Press}, year = {2007}
}

@misc{CK25,
  author = {Christopher Criscitiello and Jungbin Kim},
  title = {Horospherically Convex Optimization on {Hadamard} Manifolds {Part I}: Analysis and Algorithms},
  year = {2025},
  note = {arXiv:2505.16970},
  url = {https://arxiv.org/abs/2505.16970}
}

@misc{LLN24,
  author = {Adrian S. Lewis and Genaro L{\'o}pez-Acedo and Adriana Nicolae},
  title = {Horoballs and the subgradient method},
  year = {2024},
  note = {arXiv:2403.15749},
  url = {https://arxiv.org/abs/2403.15749}
}

@article{GLLN26,
  author = {Ariel Goodwin and Adrian S. Lewis and Genaro L{\'o}pez-Acedo and Adriana Nicolae},
  title = {Stochastic and incremental subgradient methods for convex optimization on {Hadamard} spaces},
  journal = {Mathematical Programming},
  year = {2026},
  doi = {10.1007/s10107-026-02334-4},
  note = {arXiv:2412.06730},
  url = {https://arxiv.org/abs/2412.06730}
}

@article{Hirai24,
  author = {Hiroshi Hirai},
  title = {Convex analysis on {Hadamard} spaces and scaling problems},
  journal = {Foundations of Computational Mathematics},
  volume = {24},
  pages = {1979--2016},
  year = {2024},
  doi = {10.1007/s10208-023-09628-5}
}

@article{KLM09,
  author = {Michael Kapovich and Bernhard Leeb and John J. Millson},
  title = {Convex functions on symmetric spaces, side lengths of polygons and the stability inequalities for weighted configurations at infinity},
  journal = {Journal of Differential Geometry},
  volume = {81},
  number = {2},
  pages = {297--354},
  year = {2009},
  url = {https://arxiv.org/abs/math/0311486}
}

@article{CDHS20,
  author = {Yair Carmon and John C. Duchi and Oliver Hinder and Aaron Sidford},
  title = {Lower bounds for finding stationary points {I}},
  journal = {Mathematical Programming},
  volume = {184},
  pages = {71--120},
  year = {2020},
  doi = {10.1007/s10107-019-01406-y}
}

@misc{Li26,
  author = {Zhiyuan Li},
  title = {Lecture 12: Oracle Complexity Lower Bounds},
  year = {2026},
  howpublished = {TTIC 31070: Convex Optimization, Spring 2026. Course notes},
  url = {https://zhiyuanli.ttic.edu/convex-optimization-site/notes/lecture12.pdf}
}

@misc{JK24,
  author = {Vladimir Ja{\'c}imovi{\'c} and David Kalaj},
  title = {Conformal and holomorphic barycenters in hyperbolic balls},
  year = {2024},
  note = {arXiv:2410.02257},
  url = {https://arxiv.org/abs/2410.02257}
}

@inproceedings{BCCD23,
  author = {Bonet, Cl{\'e}ment and Chapel, Laetitia and Drumetz, Lucas and Courty, Nicolas},
  title = {Hyperbolic Sliced-{Wasserstein} via Geodesic and Horospherical Projections},
  booktitle = {Proceedings of the 2nd Annual Workshop on Topology, Algebra, and Geometry in Machine Learning},
  series = {Proceedings of Machine Learning Research},
  volume = {221},
  pages = {334--370},
  year = {2023},
  publisher = {PMLR},
  url = {https://proceedings.mlr.press/v221/bonet23a.html}
}

@phdthesis{CriscitielloThesis25,
  author = {Criscitiello, Christopher Arnold},
  title = {{On the Complexity of Optimization: Curved Spaces and Benign Landscapes}},
  school = {{\'E}cole polytechnique f{\'e}d{\'e}rale de Lausanne},
  year = {2025},
  type = {PhD thesis},
  number = {11173},
  note = {Thesis no.\ 11173},
  doi = {10.5075/epfl-thesis-11173}
}

@misc{WuEtAl26,
  author = {Wendao Wu and Haihan Zhang and Chenheng Zhang and Yanyi Li and Chunyuan Zheng and Cong Fang and Haoxuan Li and Zhouchen Lin},
  title = {Geodesically Convex Optimization in Hyperbolic Space: Matching Lower Bounds},
  year = {2026},
  note = {arXiv:2610.12184},
  url = {https://arxiv.org/abs/2610.12184}
}
\end{document}